\documentclass[final]{siamart1116}

\usepackage{amsfonts,amsmath,amsbsy,amssymb,bm}
\usepackage{graphicx}
\graphicspath{{./figures/}}
\makeatletter
\def\input@path{{../}{./}}
\makeatother

\ifpdf
  \DeclareGraphicsExtensions{.eps,.pdf,.png,.jpg}
\else
  \DeclareGraphicsExtensions{.eps}
\fi

\definecolor{myBlue}{rgb}{0.0,0.0,0.55}
\usepackage{hyperref}
\hypersetup{
	plainpages=false,
    colorlinks,
    linktocpage=true,   % makes the page number as hyperlink in table of content
    linkcolor=myBlue,
    citecolor=myBlue,
    urlcolor=myBlue
}

\newtheorem{remark}[theorem]{Remark}

\newcommand{\dx}{\,{\rm d}x}
\newcommand{\dd}{\,{\rm d}}
\newcommand{\bs}{\boldsymbol}

\newcommand{\norm}[1]{\left\Vert#1\right\Vert}
\newcommand{\snorm}[1]{\left\vert#1\right\vert}
\makeatletter
\newcommand{\tnorm}{\@ifstar\@tnorms\@tnorm}
\newcommand{\@tnorms}[1]{%
\left|\mkern-2.5mu\left|\mkern-2.5mu\left|
#1
\right|\mkern-2.5mu\right|\mkern-2.5mu\right|
}
\newcommand{\@tnorm}[2][]{%
\mathopen{#1|\mkern-2.5mu#1|\mkern-2.5mu#1|}
#2
\mathclose{#1|\mkern-2.5mu#1|\mkern-2.5mu#1|}
}
\makeatother

\usepackage{comment,enumerate,multicol,xspace}

  \newcounter{mnote}
  \let\oldmarginpar\marginpar
    \renewcommand\marginpar[1]{\-\oldmarginpar[\raggedleft\footnotesize #1]%
    {\raggedright\footnotesize #1}}

\numberwithin{theorem}{section}
\numberwithin{equation}{section}

\headers{Anisotropic Error Estimates of The Linear VEM}{Shuhao Cao and Long Chen}

\title{{Anisotropic Error Estimates of The Linear Virtual Element Method on 
Polygonal Meshes
}
\thanks{The authors are supported by the National Science 
Foundation Grant No. DMS-1418934.}}

\author{
  Shuhao Cao\thanks{Department of Mathematics, University of California Irvine, 
  Irvine, CA 92697 (\email{scao@math.uci.edu}, \email{chenlong@math.uci.edu}).}
  \and
  Long Chen\footnotemark[2]
}

\usepackage{amsopn}

\ifpdf
\hypersetup{
  pdftitle={Anisotropic Error Estimates of The Lowest Order VEM},
  pdfauthor={Shuhao Cao and Long Chen}
}
\fi

\begin{document}

\maketitle

% REQUIRED
\begin{abstract}
A refined a priori error analysis of the lowest order (linear) virtual 
element method (VEM) is developed for approximating a model two dimensional Poisson 
problem. A set of new geometric assumptions is proposed on the shape regularity of 
polygonal meshes. A new universal error equation for the lowest order 
(linear) VEM is derived for any choice of stabilization, and a new stabilization 
using broken half-seminorm is introduced to incorporate short edges 
naturally into the a priori error analysis on isotropic elements. The error analysis
is then extended to a special class of anisotropic elements with high
aspect ratio originating from a body-fitted mesh generator, which uses 
straight lines to cut a shape regular background mesh. Lastly, some commonly used 
tools for triangular elements are revisited for polygonal elements to give an 
in-depth view of these estimates' dependence on shapes.
\end{abstract}

% REQUIRED

% REQUIRED

\section{Introduction}
\label{sec:intro}
To present the main idea, consider the weak formulation of the Poisson equation with 
zero Dirichlet boundary condition in a bounded two-dimensional Lipschitz domain 
$\Omega$: given an $f\in L^2(\Omega)$, find $u\in H_0^1(\Omega)$ such 
that
\begin{equation}\label{eq:weakform}
a(u,v) :=(\nabla u, \nabla v) = (f, v)\quad \forall v\in H_0^1(\Omega).
\end{equation}

To approximate problem~\eqref{eq:weakform}, $\Omega$ is decomposed into a sequence 
of polygonal meshes $\{ \mathcal T_h, h\in \mathcal H \}$. Every $\mathcal T_h$ 
consists of a finite number of simple polygons (i.e., open simply connected sets 
with non-self-intersecting polygonal boundaries). A finite dimensional approximation 
problem using virtual element method (VEM) is built upon $\mathcal{T}_h$. 
Here the subscript $h$ is the conventional notation for the mesh size: the maximum 
diameter of polygons in a mesh. The index set $\mathcal H$ contains a sequence of 
$\{h_n\}$ and $\lim_{n\to \infty}h_n = 0$. We are interested in the 
theoretical proof that the rate of convergence, of a VEM approximation 
measured under certain norm, is of what order of $h$ and the robustness to the 
geometry of the polygonal meshes.

VEM was first introduced in~\cite{Brezzi2013basic}.
Earlier error analyses of VEM 
(e.g.,~\cite{Brezzi2013basic,Mora;Rivera;Rodriguez:2015virtual})
assumed the shape regularity of the mesh, among which the most used assumptions are 
(1) every polygonal element is star-shaped with a uniform constant and (2) no short 
edge. 
VEM splits the approximation (local projections) and stability of the method into 
two terms. The star-shape assumption is mainly for the approximation property and no 
short edge is for the stability. Recently some VEM error analyses 
(see \cite{Beirao2017stability,Brenner;Sung:2018VEM}) established the stability of 
different 
choices of stabilization without the ``no short edge'' assumption. 
Moreover,~\cite{Chen;Huang:2018VEM} proposed an alternative way to perform the error 
analysis through a ``virtual'' triangulation, the idea of which can be traced back 
to regular decomposition condition in~\cite{Brezzi2009mimetic}. Further discussions 
about geometric conditions on polygons will be featured in Section \ref{sec:mesh}.

Meanwhile in many numerical tests, VEM performs robustly regardless of these 
seemingly artificial geometric constraints. For 
instances, it is shown numerically in~\cite{Beiraodaveiga2017high} that VEM 
converges in the optimal order on Voronoi meshes of which the control vertices are 
randomly generated. VEM's convergence is optimal irrelevant to the mesh cuts from 
the interfaces or fractures (see 
\cite{Benedetto;Berrone;Pieraccini;Scialo:2014virtual, Chen2017interface}). 
In~\cite{berrone2017orthogonal}, the authors illustrate that, 
locally on irregular elements, using the orthogonal polynomials, to 
which the local VEM space is projected, cures the globally 
instability arising from poor projection matrix conditioning of the 
conventional-scaled monomials with higher degrees. Similarly in 
\cite{Mascotto:2017conditioning,Mascotto;Dassi:2017Exploring}, VEM passes the 
so-called ``collapsing polygon test'' on certain types of irregular elements with 
aspect ratio unbounded, when choosing a good set of bases for the polynomials.

The purpose of this paper is to try to further develop the error analyses, inspired 
by previous references, fit for a broader class of polygonal meshes, thus justifying 
VEM's robustness.

First of all, we propose a new stabilization generalizing the scaled nodal 
difference originally proposed in~\cite{Wriggers2016virtual}. The inner product, 
inducing the broken $1/2$-seminorm on element boundaries, is used as the 
stabilization. This stabilization is elegantly simple for the linear VEM, 
which is the scenario of interest in this paper: for $u$, $v$ in the 
local VEM space in~\eqref{VEMspace} 
\begin{equation}\label{eq:difference}
S_{\mathcal{E}_K}(u, v) :=  \sum_{e\subset \partial K} 
\bigl(u(\bm{b}_e) - u(\bm{a}_e)\bigr) \bigl(v(\bm{b}_e) - v(\bm{a}_e)\bigr),
\end{equation}
in which $\bm{b}_e$ and $\bm{a}_e$ stand for the two end points of an edge $e$. The 
analysis based on~\eqref{eq:difference} matches well with an element boundary 
integral constraint for the local projection proposed in~\cite{Ahmad2013equivalent}. 

To allow a systematic anisotropic error analysis, we then propose a set of more 
``local-oriented'' assumptions inspired by the ones given 
in~\cite{Acosta2000error,Babuska.I;Aziz.A1976,Hannukainen2012,Wang2014wgmixed}. 
In~\cite{Wang2014wgmixed}, one assumption on geometrical properties of the polygonal 
meshes assumes the existence of a shape regular triangle 
based on each edge $e$ with height $l_e$, the size of which is uniform and 
comparable to the size of the underlying element. This type of assumption however 
rules out elements with high 
aspect ratio (anisotropy) and/or short edges. The reason is that an $L^2$-weighted 
trace inequality is used every time, when proving the error estimates involving 
stabilization with an $h_e^{-1}$ factor on the boundary of an element. We shall 
re-examine various trace inequalities separating the height $l_e$ from the 
length of the edge $h_e$, and propose a local height condition. As a result, 
short edge poses no problem as long as $l_e\geq c h_e$ for the error analysis on 
isotropic elements. Notice that the isotropic 
element class, under the conditions we propose in Section \ref{sec:mesh}, includes a 
more variety of polygons than the ones under the uniform star-shaped condition.

To be able to apply various trace inequalities on a long edge supporting a short 
height toward the underlying element, i.e., $l_e\ll h_e$, we shall embed an 
anisotropic element into a shape regular agglomerated element, 
e.g., a square of size $h_e$, and apply the trace inequalities in the direction of 
its shape regular neighboring element. Similar types of embedding can be traced back 
to \cite{Hannukainen2012} for finite element method, in that the convergence is 
intact 
if the approximation space on a coarse mesh, which satisfies the maximum angle 
condition~\cite{Babuska.I;Aziz.A1976}, is a subspace of the one defined on a fine 
but anisotropic mesh. However, 
for VEM a straightforward subspace argument is not valid, since the VEM spaces are 
not nested on meshes refined or coarsened from one another. 

For the approximation property of the polynomial projection, we propose an alternate 
shape regularity requirement of the mesh, utilizing the convex hull of a possible 
nonconvex element. For the interpolation error estimate, as the 
error measured in $H^1$-seminorm is transferred into the combination of that on 
boundary edges, pairing the edges will bring further cancelation and gives scales in 
different directions. For interpolation error that stems 
from the stabilization term, the nodal value difference~\eqref{eq:difference} of a 
linear polynomial (a linear VEM function restricted on boundaries) gives a 
tangential vector of this edge, which again leads to scales in different directions. 

To reveal the possible anisotropy 
arising from the VEM error analysis while fully exploiting these proof mechanism 
mentioned above, and to create a presentation easy to follow, we focus on 
the lowest order (linear) virtual element. We perform the \textit{a priori} error 
analysis directly on a weaker norm induced by the bilinear form, by writing out 
an identity of the error equation. In this way, we avoid to 
prove the norm equivalence for a stronger $H^1$-seminorm, which could 
possibly introduce more restrictive geometric constraints. Moreover, for 
the anisotropic error analysis, we restrict 
ourselves on a special class of polygonal meshes, originally obtained 
from cutting a rectangular (or a shape regular) background mesh using a set of 
straight lines (e.g.,~\cite{Chen2017interface}).

This paper is organized as follows: In Section \ref{sec:VEM}, the basics of VEM are 
covered, together with a new error equation and an \textit{a priori} error 
bound. Section \ref{sec:mesh} discusses the common geometric assumptions used in 
polygonal finite element literature, and proposes a new set of assumptions. In 
Section \ref{sec:iso}, we present the error analysis on isotropic element based on 
the new assumptions. In Section \ref{sec:aniso}, we study one class of a possible 
anisotropic element originated from a body-fitted mesh generator. Lastly in 
Appendices A-B, we revisit some conventional tools used in finite element, to learn 
the possible impacts from anisotropy on constants of some widely used inequalities, 
thus improving the error 
analysis. 

Throughout this paper, for a bounded Lipschitz domain $D$, we opt for the common 
notation $\norm{\cdot}_{0,D}$ and $(\cdot,\cdot)_D$ to denote the $L^2$-norm and 
$L^2$-inner product, respectively, and $|\cdot|_{s,D}$ to denote the 
$H^s(D)$-seminorm. When 
$D = \Omega$ is the whole domain, the subscript $\Omega$ will be omitted. 
The convex hull of $D$ is denoted as ${\rm conv}(D)$. 
For any element $K\in \mathcal{T}_h$, let $h_K = {\rm 
diam}(K)$, $|K| = \operatorname{meas}_2(K)$. Let $\mathcal N(K):= \{\bs a_1, \bs 
a_2, \ldots, \bs a_{n_K}\}$ be the set of vertices on $K$. For any edge $e\subset 
\partial K$, $h_e = {\rm diam}(e) = |e|$, $\bm{n}_e$ is the unit outward normal 
vector of $e$ to $K$, and the collection 
of the edges is $\mathcal{E}_K: = \{e\}_{e\subset \partial K}$, the subscripts 
$e$ and $K$ of which in the context may be omitted for simplicity. 
$n_{\mathcal{E}_K} := |\mathcal{E}_K|$ denotes the total number 
of edges on the boundary of an element $K$, and it equals $n_K$ which is the number 
of vertices in $K$. For any $L^1$-integrable 
function or vector field $v$, average of $v$ over the domain $D$ is denoted as 
$\overline{v}^{D}:=|D|^{-1}\int_{D}v$. The tangential derivative for 
any sufficiently regular function along a continuous curve $\Gamma$ is denoted as 
$\partial_{\Gamma} v:= \nabla v\cdot\bm{t}_\Gamma$, where $\bm{t}_\Gamma$ is the 
unit vector tangential to $\Gamma$ with a given orientation.

For convenience, $x \lesssim y$ and $z \gtrsim w$ are used to represent $x \leq c_1 
y$ and $z \geq c_2 w$, respectively, where $c_1$ and $c_2$ 
are two constants independent of the mesh size $h$. If these
constants depend on specific geometric properties of the domain that the underlying
quantities are defined on, which may happen, then such dependences, when they 
exist, shall be stated explicitly. Similarly, $a \eqsim b$ means  $a \lesssim b$ and 
$a \gtrsim b$.

\section{VEM}
\label{sec:VEM}
In this section we first introduce the linear virtual element discretization. We 
then derive a universal error equation for the difference of the VEM approximation 
to the nodal interpolation for any choice of stabilization, and present an \textit{a 
priori} error bound based on the new stabilization.

\subsection{Virtual element spaces}
The following local VEM space is introduced on a polygonal element $K$:
\begin{equation}
\label{VEMspace}
V_K := \{ v\in H^1(K): v|_{\partial K} \in B_1(\partial K), \;\Delta v =0 
\text{ in } K\},
\end{equation}
where the boundary space $B_1(\partial K)$ is defined as
\begin{equation}\label{VEMbdspace}
B_p(\partial K):= \{ v\in C^0(\partial K): v|_e\in \mathbb  P_p(e) \text{ for all } 
e\subset \partial K\},
\end{equation}
with $\mathbb  P_p(D)$ being the space of the polynomials of degree $\leq p$ defined 
on a domain $D$. 

Using the vertex values as degrees of freedom (d.o.f.), the local space is 
unisolvent (see \cite{Brezzi2013basic}), and the canonical interpolation 
in $V_K$ of $v\in H^1(K)\cap C^0(\overline{K})$ is defined as
\begin{equation}
\label{eq:intpV}
v \mapsto v_I\in V_{K} \; \text{ and }\;v_I(\bs a_i) = v(\bs a_i), \quad   \forall 
\bs a_i\in \mathcal  
N(K). 
\end{equation}
The global $H^1$-conforming virtual element space $V_h$ for problem 
\eqref{eq:weakform} is then defined as 
$V_h := \{v\in H_0^1(\Omega):\; v|_K\in V_K \}$.
The nodal interpolation of a function $v\in H^1_0(\Omega)\cap 
C^0(\overline{\Omega})$ is 
denoted by $v_I$, where $v_I\big\vert_{K} = (v|_K)_I$.

A basis of $V_K$ does not have to be represented explicitly in the computation 
procedure, in which the novelty is that the d.o.f.s are enough to produce an 
accurate and stable approximation. 

\subsection{Discretization using virtual element spaces}
The local bilinear form $a_{K}(u,v): = (\nabla u,\nabla v)_K$ for $u,v\in V_K$ 
cannot be computed exactly, since we do not know $u$ nor $\nabla u$ explicitly. 
Instead we shall compute an approximation. 

Define $\Pi_K^{\nabla}: V_K\to \mathbb  P_1(K)$ as a local projection in 
$H^1$-seminorm: given $v \in V_K$, define $\Pi_K^{\nabla} v \in \mathbb  P_1(K)$, 
which can be computed using the d.o.f. of $v$, such that
\begin{equation}\label{eq:pr}
\left(\nabla \bigl( v - \Pi_K^{\nabla} v \bigr), \nabla p\right)_K = 0, \quad 
\text{for all } p \in \mathbb  P_1(K).
\end{equation}
Note that by definition $\nabla \Pi_K^{\nabla} v = \overline{\nabla v}^K$. The 
constant kernel space of $\Pi_K^\nabla$ is then
eliminated by an extra constraint in the original VEM paper~\cite{Brezzi2013basic}:
$\sum_{i=1}^{n_K} \chi_i(v - \Pi_K^{\nabla} v) = 0$.
Taking into account the possible anisotropic nature of the element, inspired by 
\cite{Ahmad2013equivalent}, we shall use the following constraint 
\begin{equation}
\label{eq:constraint}
\int_{\partial K} \left(v -\Pi^{\nabla}_K v \right ) \dd s= 0.
\end{equation} 
Henceforth, we shall denote the local projection $\Pi^{\nabla}_K$ defined 
by~\eqref{eq:pr}--\eqref{eq:constraint} 
simply by $\Pi$, and when the projection is used in a 
global term on $\Omega$ or $\mathcal{T}_h$, it is piecewise-defined on each 
element $K$. If in a certain context this notation may give rise to ambiguity, the 
domain, upon which the projection is performed, will be explicitly mentioned.

Thanks to the local orthogonal projection~\eqref{eq:pr} in $H^1$-seminorm, the local 
continuous bilinear form $a_{K}(u,v)$ for $u,v\in V_K$ has the following split:
\begin{equation}\label{aKpi}
a_K(u,v) = \underbrace{a_K\bigl(\Pi u, \Pi  v\bigr)}_{(\mathfrak{c})} +
\underbrace{a_{K}\bigl(u-\Pi  u,v- \Pi v\bigr)}_{(\mathfrak{s})}.
\end{equation}
The \hyperref[aKpi]{$(\mathfrak{c})$} part of~\eqref{aKpi} is now explicitly 
computable, and is a decent approximation of $a_K(u, v)$. Yet it alone does not 
lead to a stable method. 

A stabilization term $S_K(\cdot,\cdot)$, matching the 
\hyperref[aKpi]{$(\mathfrak{s})$} part in~\eqref{aKpi} yet computable, will be added 
to gain the coercivity of the discretization. Therefore VEM is in fact a family of 
schemes different in the choice of stabilization terms. Define 
\begin{equation}
\label{eq:bl-vem}
a_h(u,v) :=\sum_{K\in \mathcal  T_h} a_K(\Pi u, \Pi v) + 
\sum_{K\in \mathcal  T_h} 
S_{K}\bigl(u-\Pi u,v- \Pi v\bigr).
\end{equation}
A VEM discretization of~\eqref{eq:weakform} is: find $u_h\in V_h$ such that
\begin{equation}\label{eq:VEMsol}
a_h( u_h, v_h) = (f, \Pi v_h) \qquad \forall v_h \in V_h.
\end{equation}
Notice that in the right-hand side, we opt to use the $H^1$-projection $\Pi v_h$ not 
the $L^2$ projection which is not computable for the linear VEM.

The principle of designing a stabilization is two-fold~\cite{Brezzi2013basic}: 
(1) Consistency. $S_K(u,v)$ should vanish when either $u$ or 
$v$ is in $\mathbb P_1(K)$. This is always true as in \eqref{eq:bl-vem} the slice 
operator $(\operatorname{I}-\Pi)$ is applied to the inputs of $S_K(\cdot,\cdot)$ 
beforehand. (2) Stability and continuity. $S_K(\cdot,\cdot)$ is chosen so that the 
following 
norm equivalence holds
\begin{equation}\label{eq:normequivalence}
a(u, u) \lesssim a_h(u, u) \lesssim a(u, u)  \qquad u\in V_h.
\end{equation}
Thorough error analysis based on~\eqref{eq:normequivalence} of several stabilization 
terms under certain geometric assumption on $\mathcal{T}_h$ can be found in 
\cite{Beirao2017stability,Brenner;Sung:2018VEM}.

In view of the orthogonal decomposition, i.e., setting $u=v$ in~\eqref{aKpi}, and 
the fact that $u-\Pi u$ is harmonic in $K$, the ideal choice of the stabilization 
would be the inner product that induces $1/2$-seminorm on $\partial K$: 
$(u-\Pi u, v-\Pi v)_{\frac{1}{2}, \partial K}$. Since the stabilization term then 
replicates the second term in the split~\eqref{aKpi}:
\begin{equation}\label{eq:minnorm}
|u-\Pi u|_{\frac12,\partial K} = \norm{\nabla (u-\Pi u)}_{0,K}.
\end{equation}
The identity~\eqref{eq:minnorm} is indeed a definition of the $1/2$-seminorm as 
the a quotient type norm using a harmonic extension and a trace theorem 
(\cite[Chapter 1, section 8]{Lions-Magenes}). We shall choose a computable 
definition of $1/2$-seminorm, which is equivalent 
with the one in~\eqref{eq:minnorm}: for any 
function $v$ defined locally on a hyperplane $\Gamma$ in $\mathbb{R}^2$, 
\begin{equation}
\label{eq:norm-half}
\snorm{v}_{\frac12,\Gamma}^2:= 
\int_{\Gamma} \int_{\Gamma} \frac{|v(\bm{x}) - v(\bm{y})|^2}{|\bm{x}-\bm{y}|^{2}} 
\dd s(\bm{x}) \dd s(\bm{y}).
\end{equation}
Here $\Gamma$ can be the whole $\partial K$ or a connected 
component being a subset of $\partial K$.

The inner product that induces $1/2$-seminorm on $\partial K$ is further decomposed 
into a broken one to be used as the stabilization $S_K(\cdot,\cdot)$ in 
\eqref{eq:bl-vem}, that is, 
\begin{equation}\label{eq:bilinear-brokenhalf}
S_{\mathcal{E}_K}\bigl(u, v\bigr) 
:= \sum_{e\subset \partial K}(u, v)_{\frac12, e},
\qquad u, v \in ( \operatorname{I} - \Pi)V_K,
\end{equation}
where on an edge $e$
\begin{equation}
(u, v)_{\frac12, e}: = \int_{e} \int_{e} \frac{\bigl(u(\bm{x}) - 
u(\bm{y})\bigr) \big( v(\bm{x}) - v(\bm{y}) \big)}{|\bm{x}-\bm{y}|^{2}} 
\dd s(\bm{x}) \dd s(\bm{y})
\end{equation}
For the linear VEM, as the function restricted on each edge is linear, the stabilization term features a very simple formula as follows:
\begin{equation}
\label{eq:difference-half}
S_{\mathcal{E}_K}(u, v) =  \sum_{e\subset \partial K} 
\bigl(u(\bm{b}_e) - u(\bm{a}_e)\bigr) \bigl(v(\bm{b}_e) - v(\bm{a}_e)\bigr),
\quad u, v\in V_K,
\end{equation}
in which $\bm{b}_e$ and $\bm{a}_e$ stand for the two end points of edge $e$.
We remark that this stabilization, in the case of linear VEM, is equivalent to a variant of the scaled nodal difference originally proposed in~\cite{Wriggers2016virtual} but philosophically different.

Note that by definition, $u\in H^{\frac12}(\partial K)$ measured under the 
broken $1/2$-seminorm induced by $S_{\mathcal{E}_K}(\cdot,\cdot)$ is always 
bounded by its $1/2$-seminorm on the whole boundary:
\begin{equation}
\label{eq:brokenhalfnorm}
|u|_{\frac12,\mathcal{E}_K}^2:= S_{\mathcal{E}_K}(u, u)= 
\sum_{e\subset \partial K}|u|_{\frac12,e}^2 \leq |u|_{\frac12, \partial K}^2. 
\end{equation}
However, the following inequality of the reverse direction, 
\begin{equation}\label{eq:broken2all}
|u|_{\frac12, \partial K}^2 \lesssim  \sum_{e\subset \partial K} 
|u|_{\frac12,e}^2=S_{\mathcal{E}_K}(u, u),
\end{equation}
which is usually a key component of the norm equivalence in the traditional VEM 
error analysis, is in general not true for an arbitrary function in 
$H^{\frac12}(\partial K)$ (e.g., see~\cite[Theorem 1.5.2.3]{Grisvard1985}). This 
reverse inequality~\eqref{eq:broken2all} does hold for continuous and piecewise 
polynomials~\cite{Cao2004pre}, but the constant hidden in~\eqref{eq:broken2all} 
depends on the geometry of $K$ and is not fully characterized. 

If we follow the classical approach of error analysis of VEM (see, 
e.g.,~\cite{Brezzi2013basic,Beirao2017stability}), by showing the 
norm equivalence~\eqref{eq:normequivalence}, geometric conditions on polygons such 
as shape regularity are indispensable in the proof mechanism. Instead, we shall 
present a different approach by writing out an error 
equation.

\subsection{Error equation}
In this subsection an error equation for $u_h - u_I$ is derived, where $u_h$ is the 
solution to the VEM approximation problem~\eqref{eq:VEMsol} and $u_I$ is the nodal 
interpolation~\eqref{eq:intpV}. A weaker mesh-dependent norm 
induced by the bilinear form~\eqref{eq:bl-vem}, with 
$S_{\mathcal{E}_K}(\cdot, \cdot)$ as the stabilization on each element $K$, 
can be defined as follows and shall be a key ingredient in our analysis:
\begin{equation}
\label{def:norm-t}
\tnorm{u}:= a_h^{1/2}(u,u) = \left\{\sum_{K\in \mathcal T_h}
\Big(\left\|\nabla \Pi  u \right\|_{0,K}^2 + 
\left|  u - \Pi u \right|_{\frac12,\mathcal{E}_K}^2 
\Big)\right\}^{1/2}.
\end{equation}

\begin{lemma}[a mesh-dependent norm]
 $\tnorm{\cdot}$ defines a norm on $V_h$.
\end{lemma}
\begin{proof}
For $u\in V_K$, as $u - \Pi u$ is harmonic in $K$, $u - \Pi u$ 
can be treated as $(u - \Pi u)\big|_{\partial K}$'s minimum norm 
extension in $|\cdot|_{1,K}$. As a result, the following estimate follows from a 
standard extension theorem (e.g.,~\cite[Theorem 2.5.7]{Necas1967},~\cite[Theorem 
4.1]{Xu1998ddm}) and the validity of~\eqref{eq:broken2all} for  
continuous and piecewise polynomials (see \cite{Cao2004pre})
\begin{equation}\label{eq:norm}
\| \nabla (u -  \Pi  u)\|_{0,K} 
\leq C_K \left| u -  \Pi u \right|_{\frac12, \partial K} 
\lesssim C_K \left| u - \Pi u \right|_{\frac12,\mathcal{E}_K}.
\end{equation}
By the splitting~\eqref{aKpi}, the coercivity $\|\nabla u\|^2 \lesssim  a_h(u, u)$ 
then holds $\forall u\in V_h$, from which we conclude that $\tnorm{\cdot}$ is a norm 
on $V_h\subset H^1_0(\Omega)$.
\end{proof}

\begin{remark}[norm equivalence]\rm Indeed~\eqref{eq:norm} implies 
$a(u,u)\lesssim a_h(u,u)$ and the equivalence of two definitions of 
$1/2$-norm,~\eqref{eq:minnorm}--\eqref{eq:norm-half} implies $a_h(u,u)\lesssim 
a(u,u)$. Therefore we obtain the norm equivalence~\eqref{eq:normequivalence} for 
$u\in V_h$. 
However, the constants involved in the norm equivalence depend on the geometry of 
$K$ and is not robust to the aspect ratio. Using $\tnorm{\cdot}$, a constant-free 
stability can be obtained. $\Box$
\end{remark}

\begin{lemma}[stability] \label{lemma:stability}
Given an $f =  -\Delta u \in L^2(\Omega)$ for some $u\in H_0^1(\Omega)$, problem 
\eqref{eq:VEMsol} is well posed and has the following constant-free stability
\begin{equation}\label{eq:stabilityidentity}
\tnorm{u_h} = \sup_{v_h \in V_h}\frac{\bigl(f, \Pi v_h\bigr)}{\tnorm{v_h}}.
\end{equation}
\end{lemma}
\begin{proof}
First $\tnorm{\cdot}$ defines a norm on $V_h$, and $\bigl(f, \Pi 
(\cdot)\bigr)$ is a linear functional on $V_h$, now since $V_h$ is finite 
dimensional, $\bigl(f, \Pi (\cdot)\bigr)$ is continuous with respect to 
$\tnorm{\cdot}$. Then the identity~\eqref{eq:stabilityidentity} follows from the 
Riesz representation theorem.
\end{proof}

The \textit{a priori} error analysis shall be carried out for $\tnorm{u_I - u_h}$ by 
writing out first the following error equation. 
\begin{theorem}[an error equation]\label{theorem:error-eq}
Let $u_h$ be the solution to~\eqref{eq:VEMsol} and $u_I$ be the nodal 
interpolation defined in~\eqref{eq:intpV}.  
For any stabilization $S_K(\cdot,\cdot)$, the following error
representation holds for any $v_h\in V_h$:
\begin{equation}
\label{eq:error}
\begin{aligned}
a_h\bigl(u_h - u_I, v_h\bigr) & = 
\sum_{K\in \mathcal{T}_h} \left( \nabla \Pi  (u- u_I), 
\nabla \Pi  v_h\right)_K 
\\
& \; - \sum_{K\in \mathcal{T}_h} S_K\left( u_I - \Pi  u_I,
v_h - \Pi  v_h \right)
\\
 & \; + \sum_{K\in \mathcal{T}_h} {\left\langle
  \nabla\big(u - \Pi  u \big)
\cdot 
\bs{n}, v_h - \Pi v_h \right\rangle}_{\partial K}.
\end{aligned}
\end{equation}
\end{theorem}

\begin{proof}
For any $v_h\in V_h$, using the VEM discretization~\eqref{eq:VEMsol} being stable 
(Lemma \ref{lemma:stability}), the underlying PDE $-\Delta u = f$, integration 
by parts element-wisely, and $\nabla u\cdot \bs n$ being continuous across 
interelement boundaries, we have:
\begin{equation}\label{eq:err-1}
\begin{aligned}
& a_h(u_h, v_h) - a_h(u_I, v_h)= -\sum_{K\in \mathcal{T}_h} \bigl(\Delta u, 
\Pi  v_h \bigr)_K - a_h(u_I, v_h)
\\
= & \;\sum_{K\in \mathcal{T}_h} \bigl(  \nabla u , \nabla \Pi  v_h \bigr)_K
- \sum_{K\in \mathcal{T}_h} \left\langle   \nabla u \cdot \bs{n}, 
\Pi  v_h \right\rangle_{\partial K} - a_h(u_I, v_h)
\\
=& \;\sum_{K\in \mathcal{T}_h} \left(   \nabla \Pi  (u- u_I), 
\nabla \Pi  v_h\right)_K - \sum_{K\in \mathcal{T}_h} S_K\left( u_I - 
\Pi  u_I,
v_h - \Pi  v_h \right)
\\
&+ \sum_{K\in \mathcal{T}_h} \left\langle   \nabla u \cdot \bs{n}, 
v_h - \Pi  v_h \right\rangle_{\partial K}.
\end{aligned}
\end{equation}
For the last boundary integral term in~\eqref{eq:err-1}, the final error 
equation~\eqref{eq:error} follows from exploiting the definition 
of the projection $\Pi = \Pi^{\nabla}_K$ in~\eqref{eq:pr} and $\Delta \Pi u=0$:
\begin{equation}
\left\langle \nabla \Pi  u \cdot \bs{n}, 
v_h - \Pi  v_h \right\rangle_{\partial K}= \big(\nabla \Pi  
u, \nabla(v_h - \Pi  v_h) \big)_K = 0.
\end{equation}
\end{proof}

We emphasize again that this is an identity for any choice of stabilization terms, 
under which the solution $u_h$ exists for problem~\eqref{eq:VEMsol}. 
To get a meaningful convergence result, however, we need a 
stabilization to be able to control the boundary term 
$\left\langle  \nabla \bigl(u-\Pi  u\bigr) \cdot \bs{n},  v_h - 
\Pi v_h \right\rangle_{\partial K}$ and meanwhile $u_I - \Pi u_I$ is of optimal 
order measured under the seminorm induced by $S_K(\cdot,\cdot)$.

\begin{corollary}[an \textit{a priori} error bound]\label{corollary:errbound}
Under the same setting with Theorem \ref{theorem:error-eq}, for the stabilization 
$S_{\mathcal{E}_K}(\cdot,\cdot)$ in~\eqref{eq:difference-half}, the following 
estimate holds:
\begin{equation}
\label{eq:err-apriori}
\begin{aligned}
\tnorm{u_h-u_I}\lesssim & \;\Bigg \{\sum_{K\in \mathcal T_h}  \Big(\|\nabla 
\Pi  (u- 
u_I)\|_{0,K}^2 + \left|u_I - \Pi u_I\right|_{\frac12,\mathcal{E}_K}^2 
\\
&\;+ n_{\mathcal{E}_K}^2 \sum_{e\subset \partial K}
h_e\norm{\nabla\big(u - \Pi u \big)\cdot \bs{n}}_{0,e}^2 \Big)
\Bigg\}^{1/2}.
\end{aligned}
\end{equation}
\end{corollary}
\begin{proof}
Let $v_h = u_h - u_I\in  V_h$ in the error representation 
\eqref{eq:error}; applying the Cauchy--Schwarz inequality on the three terms, 
respectively, yields:
\begin{equation}
\label{eq:error-1}
\begin{aligned}
\tnorm{u_h-u_I}^2 & \leq 
\sum_{K\in \mathcal{T}_h} \norm{\nabla \Pi  (u- u_I)}_{0,K}
\norm{\nabla \Pi  v_h}_{0,K} 
\\
& \; + \sum_{K\in \mathcal{T}_h} 
\left|u_I - \Pi  u_I\right|_{\frac12,\mathcal{E}_K}\,
\left|v_h - \Pi  v_h\right|_{\frac12,\mathcal{E}_K}
\\
 & \; + \sum_{K\in \mathcal{T}_h} \sum_{e\subset\partial K}
h_e^{1/2}\norm{\nabla\big(u - \Pi  u \big)\cdot \bs{n}}_{0,e}\,
h_e^{-1/2}\norm{v_h - \Pi  v_h}_{0,e}.
\end{aligned} 
\end{equation}
For the last term above, $\overline{(v_h - \Pi v_h)}^{\partial K} = 0$ by the 
constraint in~\eqref{eq:constraint}. As a result, the error bound follows from the 
Cauchy--Schwarz inequality, and applying the Poincar\'e inequality 
\eqref{eq:poincare-brokenhalf} on each $e$: 
recall that $n_{\mathcal{E}_K}$ 
is the number of edges on $\partial K$,
\begin{equation}
\label{eq:stab-shortedge}
\sum_{e\subset \partial K} h_e^{-1}\norm{v_h - \Pi  v_h}_{0,e}^2
\leq n_{\mathcal{E}_K}^2 \left|v_h - \Pi v_h\right|_{\frac12,\mathcal{E}_K}^2.
\end{equation}
\end{proof}
In later sections, we shall estimate the three terms in~\eqref{eq:err-apriori} based 
on certain geometric assumptions.

\section{Geometric assumptions on polygonal meshes}
\label{sec:mesh}
Some aforementioned error analyses of VEM (see Section \ref{sec:intro})
are based on the following assumptions on a polygonal mesh $\mathcal{T}_h$ in 
2-D cases:
\begin{enumerate}
\item[{\bf C1}.] \label{asp:C}There exists a real number $\gamma_1>0$ such that, for 
each element $K\in \mathcal {T}_h$, it is star-shaped with respect to a disk of 
radius $\rho_K\ge\gamma_1 h_K$.

\item[{\bf C2}.] There exists a real number $\gamma_2>0$ such that, for each element 
$K\in \mathcal{T}_h$, the distance between any two vertices of $K$ is $\ge \gamma_2 
h_K$.
\end{enumerate}

We shall refer to {\bf C1} as the $\gamma_1$-shape regular condition and {\bf C2} as 
the no short edge condition. Assumption {\bf C1} rules out polygons with high aspect 
ratio, which shall be called as anisotropic element. Equivalently the current error 
estimate is not robust to the aspect ratio of $K$. Assumption {\bf C2} rules out 
edges with small length which may exists, for example, in polygons of Voronoi 
tessellation (see \cite{du1999centroidal,talischi2012polymesher}). A shape regular 
polygon may have short edges. A polygon with similar edge lengths may not be shape 
regular. In particular, triangles/quadrilaterals and tetrahedra satisfying {\bf C2} 
but not {\bf C1} are known as slivers, which are problematic in finite 
element simulations. Yet in practice (see section 
\ref{sec:intro}), VEM is robust even when {\bf C1},  and/or {\bf C2} are violated.  

Next, 
we shall propose a set of geometry conditions on polygonal meshes based on the 
following local quantities for an edge. For an edge $e\subset \partial K$, we choose 
a local Cartesian coordinate system with $x$-axis aligning with $e$, and positive 
$y$-direction to be the inward normal of $e$. Define
\begin{equation}
\label{eq:local-de}
\delta_e := \inf \Big\{\delta\in \mathbb{R}^+ : K\cap \big(e\times 
(\delta,+\infty)\big) = \varnothing \Big\}.
\end{equation} 
For each edge $e\subset \partial K$, an inward height $l_e$ associated 
with edge $e$ can be defined as follows. Let $T(e,l)$ be any triangle with base $e$ 
and height $l$,
\begin{equation}
\label{eq:local-le}
l_e := \sup \Big\{l\in \mathbb{R}^+ : \exists\, T(e,l) \subset K\cap \big(e\times 
(0, \delta_e]\big) \Big\}.
\end{equation}
As $K$ is nondegenerate and bounded, $0< \delta_e < +\infty$ and $0< l_e \leq  
\delta_e$. This height $l_e$ measures how far from the edge $e$ one can advance to 
the interior 
of $K$ in its inward normal direction. The rectangle $e\times(0, \delta_e]$ is used 
to ensure that the two side angles adjacent to $e$ of triangle $T_e = T(e,l_e)$ are 
bounded by $\pi/2$. See Figure~\ref{fig:localheight}(a).

\begin{figure}[htp]
\centering
\includegraphics[width = 0.6\textwidth]{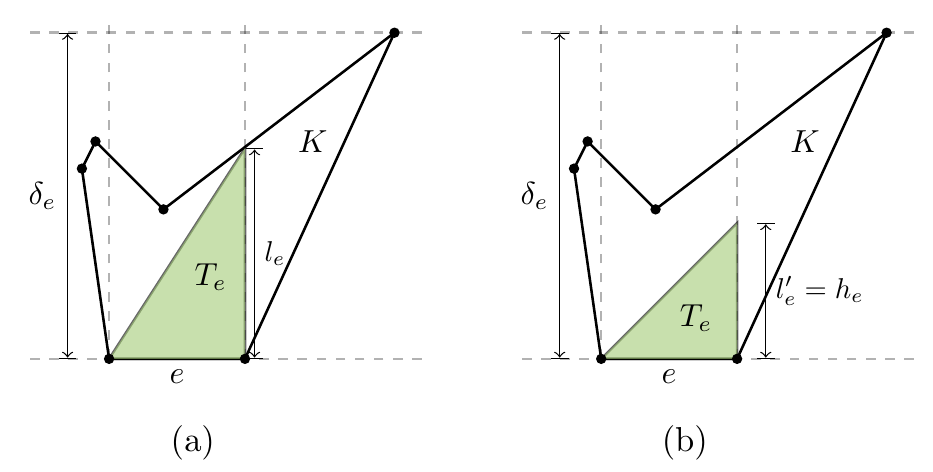}
\caption{
{\rm(a)} There exists $\gamma>1$ such that $l_e\geq \gamma h_e$. {\rm(b)} $T_e$ is 
rescaled with height $l_e' = h_e$.
}
\label{fig:localheight}
\end{figure}
\begin{remark}\rm
It would be more meaningful to use notation $h_e$ for the height and $l_e$ for 
the length of $e$. However, by the convention of finite element analysis, $h_e$ has 
been reserved for the length. $\Box$
\end{remark}

$\mathcal{T}_h$ is said to satisfy a certain assumption, if every $K\in 
\mathcal{T}_h$ satisfies that assumption with respect to its constant uniformly as 
$h\to 0$. An element $K$ is said to be {\em isotropic} if the following 
assumptions ({\bf A1-A2}) both hold for an element $K$:
\smallskip
\begin{itemize}
\item[{\bf A1.}] \label{asp:A} There exists a constant $C_1>0$ such that the number 
of edges $n_{\mathcal{E}_K} \leq C_1$.
\smallskip
\item[{\bf A2.}] There exists a constant $\gamma>0 $ such that 
$\forall e\subset \partial K$, $l_e \geq \gamma h_e$.  
\end{itemize}
\smallskip
Without loss of generality, one can assume that {\bf A2} holds with $0<\gamma\leq 1$ 
when using {\bf A2} as a premise of a certain proposition. The 
reason is that, when {\bf A2} holds, one can always rescale the height 
$l_e$ to $l'_e = \gamma' h_e$, for any $0<\gamma'\leq \gamma$, while a new $T_e 
= T(e,l'_e)$ still satisfies $T_e\subset K$. When $\gamma> 1$, we can simply set 
$\gamma'=1$ to be the new $\gamma$. See the illustration in Figure 
\ref{fig:localheight}(b).

Now the anisotropicity of an element $K$ can be characterized by: for one or more 
edges $e\subset \partial K$, either $l_e \ll h_e$ or $h_e \ll l_e$. As the upcoming 
analysis has shown, the case $h_e \ll l_e$, i.e., a short edge, is allowed as long 
as {\bf A2} holds. If $h_e \ll l_e$, one can use the rescaling argument above to 
obtain a smaller $T_e = T(e,l_e)$ so that $l_e\eqsim h_e$. In proving the error 
estimates, $l_e \geq \gamma h_e$ in {\bf A2} is needed when using trace inequalities 
(Lemmas \ref{lemma:iso-interp}, \ref{lemma:isopr-nd}, and \ref{lemma:iso-stab}) or a 
generalized Poincar\'e inequality (Lemma \ref{lemma:aniso-projnd}).  
Assumption {\bf A1} is needed, otherwise one can always artificially divide a long 
edge into short edges to satisfy {\bf A2}.

The case $l_e \ll h_e$ is difficult, because the lack of room inside the element 
makes impossible a smooth extension of functions defined on the boundary. When {\bf 
A2} is not met, it is possible to get a robust error estimate 
by embedding an anisotropic element in a shape regular one, and separating $h_e$ 
and $l_e$ in the refined error analysis.

\begin{figure}[h]
\centering
\includegraphics{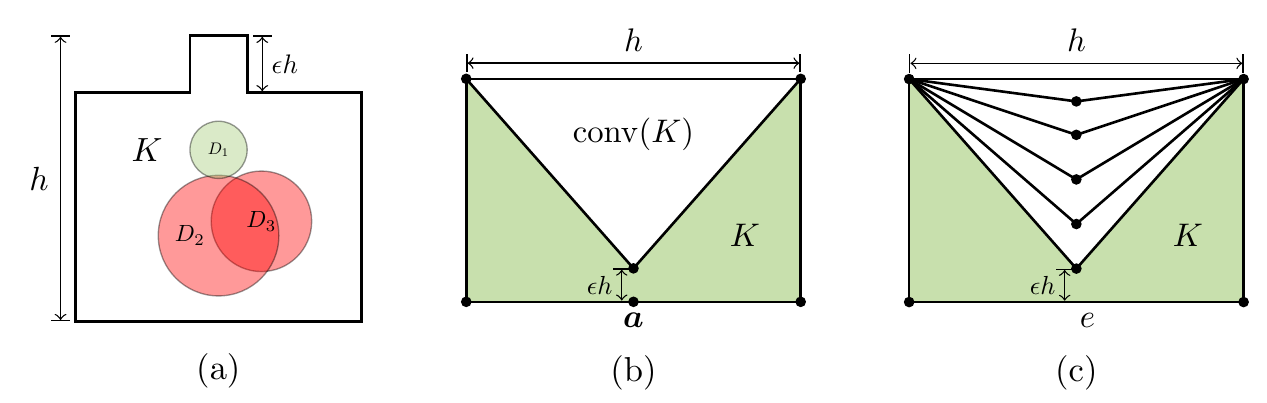}
\caption{
{\rm(a)} The isotropic $K$ is only star-shaped with respect to 
$D_1$ with 
radius $\epsilon h$, not to $D_2$ or $D_3$. {\rm(b)} The isotropic $K$ has the 
infamous hourglass shape, but VEM performs robustly on it thanks to the vertex 
$\bm{a}$ and a bounded number of elements in its convex hull $\mathrm{conv}(K)$. 
{\rm(c)} 
$K$ is anisotropic, edge $e$ violates {\bf A2} due to the lack of 
an extra vertex in the middle; meanwhile {\bf A3} is violated if the number of 
nonconvex wedges grows unbounded as $h\to 0$.
}
\label{fig:starelem}
\end{figure}
The next requirement for an element $K\in \mathcal{T}_h$, Assumption {\bf A3}, can 
be also treated as certain shape regularity of the mesh. It is needed for the 
approximation property, which is originally guaranteed by the star-shaped condition 
\hyperref[asp:C]{\bf C1}. As we do not enforce \hyperref[asp:C]{\bf C1}, we provide 
an alternative way of proving the approximation estimate of the projection in Lemma 
\ref{lemma:isopr}. For an element $K$, consider the convex hull $\mathrm{conv}(K)$ 
of $K$, and obviously $h_{K} = h_{\mathrm{conv}(K)}$. Define 
$n_{\mathrm{conv}(K)}$ to 
be 
the cardinality of the following set:
\begin{equation}\label{eq:omegaK}
n_{\mathrm{conv}(K)} := \bigl|\{K'\in \mathcal{T}_h: \; 
K \cap \mathrm{conv}(K') \neq \varnothing \} \bigr|.
\end{equation}
\begin{itemize}
\item[{\bf A3.}] \label{asp:A3} 
There exists a constant $C_3>0$ such that $n_{\mathrm{conv}(K)} \leq C_3$. 
\end{itemize}

\smallskip
$n_{\mathrm{conv}(K)}$ represents the number of times an 
element $K$ has a nonempty intersection with the convex hulls of $K$'s neighboring 
elements. $\mathcal{T}_h$ satisfies {\bf A3} if the number of polygons touching 
every vertex is uniformly bounded, which is true for the 
most popular polygonal mesh generators (see 
\cite{du1999centroidal,talischi2012polymesher}). 

In Section \ref{sec:iso}, we shall present the error analysis for isotropic 
elements, and for a special 
class of anisotropic polygons on which {\bf A1} holds but not {\bf A2} in Section 
\ref{sec:aniso}. {\bf A3} is required in both scenarios.

In~\cite{Beirao2017stability,Brenner;Sung:2018VEM}, short edges are allowed and 
integrated into the VEM error analysis by only assuming {\bf A1--C1}.
{\bf A2} is inspired by and modified from the uniform height condition $l_e\geq 
\gamma h_K$, which is posed as a shape regularity constraint in 
\cite{Wang2014wgmixed}. One can easily verify that {\bf C1} implies {\bf A2}. When the polygon is star-shaped and $\gamma$-shape regular for a uniform constant 
$\gamma$, for an edge $e$, the triangle $T_e$ can be chosen 
as the one formed by $e$ and the center of the disk. Although $T_e$ may not be shape 
regular, its height satisfies {\bf A2}. In our 
opinion, being shape regular is a local condition near an edge, while 
being star-shaped requires global information about the whole polygon $K$. In this 
sense {\bf A2} is weaker than {\bf C1}. For example, a polygon can be isotropic but 
not uniformly star-shaped; please refer to Figure \ref{fig:starelem}(a), (b), and, 
in the extreme, the pegasus polygon (winged horses) constructed from M.C Escher's 
tessellations in \cite{Brezzi2014VEM}. 

We now explore more on the geometric properties of isotropic elements. 
\begin{lemma}[scale of the polygon area for isotropic elements]
\label{lemma:iso-area}
For an isotropic element $K$, i.e., $K$ satisfies {\bf A1-A2}, we 
have the relation
\begin{equation}\label{eq:areaKh}
 n_{\mathcal E_K}^{-2}\gamma \, h_K^2/2 \leq |K| \leq \pi h_K^2.
\end{equation}
\end{lemma}
\begin{proof}
$|K| \leq \pi h_K^2$ is true by the definition of the diameter. To 
prove the lower bound, take the longest edge $e$. By {\bf A1}, $h_K \leq |\partial 
K| \leq n_{\mathcal{E}_K} h_e$. We then have $h_K^2 \leq  
n_{\mathcal{E}_K}^2h_e^2 \leq (n_{\mathcal{E}_K}^2/\gamma) l_e h_e \leq 
2(n_{\mathcal{E}_K}^2/\gamma) |T_e|\leq 2(n_{\mathcal{E}_K}^2/\gamma) |K|$.
\end{proof}

For anisotropic polygons, there might exist an edge $e$ such that $l_e \ll h_e$, 
i.e., this edge $e$ violates the height condition. The relation between the area and 
the edge length could be $l_eh_e/2 \leq |K| \ll h_e^2.$ 
%A robust error estimate should be able to separate $h_e$ with $l_e$.

It is not pragmatic to analyze anisotropic polygons of arbitrary shapes. 
Later in Section \ref{sec:aniso}, for the ease of presentation,  we will 
restrict ourself to the anisotropic polygons obtained by cutting a square with one 
straight line, which is a common practice for interface problems (e.g., 
\cite{Chen2017interface} and the references therein). Depending on the location of 
the cut, a triangle or quadrilateral may violate {\bf A2} as $h\to 0$. 

\begin{lemma}[property of polygons cut from a square]
\label{lemma:SKR}
Let $S$ be a square. Assume that $S = K\cup R$, $\bar K\cap \bar R$ is a line 
segment, and $|K|, |R| > 0$. Then both $K$ and $R$ satisfy {\bf A1} and either $K$ 
or $R$ is isotropic, i.e., satisfies {\bf A1--A2}.
\end{lemma}
\begin{proof}Obviously {\bf A1} holds with the upper bound $C=5$. 
%{\bf A1} holds as 
%the Delaunay triangulation of vertices and cut points satisfy the maximum angle 
%condition (see~\cite{Chen2017interface} Proposition 3.2).

Without loss of generality, it is assumed that $|S|= 1$ and $R$ is the polygon with 
larger area. There are only two cases. 

$\bullet$ Case 1: two cut points are on two neighboring sides of $S$ and $R$ is a 
pentagon. 

$\bullet$ Case 2: two cut points are on the opposite sides of $S$ and $R$ is a 
trapezoid and could be degenerate to a triangle when the cut forms 
the diagonal of $S$. 

For any edge $e$ of $R$, one can choose the farthest vertex, which has the largest 
distance to this edge to form the triangle $T_e$, and the height $l_e$ is the 
distance. The maximum angle of $T_e$ is bounded by $\pi/2$. 
For {\bf A2}: (1) When the edge $e$ is also on the boundary of the 
square, the height $l_e$ is the distance of the vertex to the 
edge $\geq 1/2$ as $R$ is assumed to have the larger area. (2) For the 
cut edge, the largest distance is $\geq \sqrt{2}/2$ in the pentagon case, and $\geq 
2/\sqrt{5}$ for the trapezoid.
\end{proof}

\section{Error analysis for isotropic elements}
\label{sec:iso}
In this section we shall provide optimal order error estimate for linear VEM 
approximation on isotropic elements, i.e., polygons satisfying Assumption 
\hyperref[asp:A]{\bf A1--A2--A3}. 

We first present an interpolation error estimate. Thanks to the error bound 
\eqref{eq:err-apriori}, we do not need to estimate $\norm{\nabla(u - u_I)}$ but its 
projection $\norm{\nabla \Pi(u - u_I)}$, which can be transferred to the element 
boundary through integration by parts. The interpolation error $u-u_I$ is well 
understood along the boundary. \hyperref[asp:A]{\bf A2} gives room to trace 
inequalities to lift the estimate to the interior. 

\begin{lemma}[a projection-type error estimate for the interpolation on isotropic 
polygons]\label{lemma:iso-interp}
If $K$ satisfies \hyperref[asp:A]{\bf A1--A2} with $0<\gamma\leq 1$, then 
for any $u\in H^2(K)$,
\begin{equation}
\label{est:isointerp}
\norm{\nabla \Pi  (u - u_I)}_{0,K} \lesssim 
n_{\mathcal{E}_K} \gamma^{-1} \, h_K |u|_{2,K}.
\end{equation}
\end{lemma}
\begin{proof}
Let $q = \Pi  (u - u_I) \in \mathbb{P}_1(K)$, then by the definition 
of $\Pi $ in~\eqref{eq:pr}:
\begin{equation}
\left(\nabla \Pi  (u - u_I), \nabla q \right)_K 
= \left(\nabla (u - u_I), \nabla q \right)_K 
= \sum_{e\subset \partial K} (u - u_I, \nabla q \cdot \bs{n})_e.
\end{equation}
Being restricted to one edge, $(u - u_I)|_e$ can be estimated by the standard 
interpolation error estimate in fractional Sobolev norm (see e.g. 
\cite[Section 8, example 3]{Dupont-Scott1980} and \cite{Ern2017finite}):
$\norm{u - u_I}_{0,e}\lesssim h_e^{3/2} |u|_{\frac32,e}$. For the edge $e$ 
satisfying the \hyperref[asp:A]{\bf A1--A2}, use trace inequality 
\eqref{eq:tr-1/2} in Lemma \ref{lemma:tr-a} on $\nabla u$ component-wisely for the 
term $|u|_{\frac32,e}$,
\begin{equation} 
\label{eq:intp-1}
\begin{aligned}
 &(u - u_I, \nabla q \cdot \bs{n})_e 
\leq \norm{u - u_I}_{0,e} \norm{\nabla q\cdot \bs{n}}_{0,e}
\lesssim  \; h_e^{3/2} |u|_{\frac32,e}\, |\nabla q| \,h_e^{1/2}\\
\lesssim &\; \frac{1}{\sqrt{\gamma}} 
\frac{h_e}{|K|^{1/2}}\, h_e |u|_{2,K} \,|\nabla q| |K|^{1/2}\,
\lesssim 
\;  \frac{1}{\gamma} h_K |u|_{2,K} \norm{\nabla q}_{0,K},
\end{aligned}
\end{equation}
wherein the last step, we have used $h_e\leq h_K$ and $h_e^2/ |K| \lesssim 
\gamma^{-1}l_e h_e/|K|\leq \gamma^{-1}$ implied by \hyperref[asp:A]{\bf A2}. Summing 
up on each $e$ and canceling $\norm{\nabla q}_{0,K}$, we get~\eqref{est:isointerp}.
\end{proof}

In view of the proof, the error contribution is actually proportional to the edge 
length and thus a short edge is not an issue. \hyperref[asp:A]{\bf A2} is 
required to apply the trace inequalities, as well as the ratio $h_e/|K|^{1/2}$ 
being bounded. \hyperref[asp:A]{\bf A1} is needed in that the error 
estimate in~\eqref{eq:intp-1} is summed over all edges. For anisotropic elements, 
a long edge 
may only support a very short height moving inward to the interior of $K$ and thus 
$|K|\ll h_e^2$, needing a more delicate analysis detailed in section 
\ref{sec:aniso}. We now move to the estimate of the projection error $\nabla(u - 
\Pi  u)$.  

\begin{lemma}[error estimate for the projection]\label{lemma:isopr}
Let $\mathrm{conv}(K)$ be the convex hull of $K$, then for any $u\in H^2({\rm 
conv}(K))$, the 
following error estimate holds:
\begin{equation}\label{eq:piH1}
\norm{\nabla \bigl(u - \Pi u \bigr)}_{0,K}\leq 
\frac{h_K}{\pi}|u|_{2,\mathrm{conv}(K)}.
\end{equation}
\end{lemma}
\begin{proof}
As $\nabla \Pi  u = \overline{\nabla u}^K$ is the best constant approximation in 
$L^2(K)$-norm:
\begin{equation*}
\norm{\nabla \bigl(u - \Pi  u \bigr)}_{0,K}
\leq \big\| \nabla u - \overline{\nabla u}^{\mathrm{conv}(K)}\big\|_{0,K}
\leq \big\| \nabla u - \overline{\nabla 
u}^{\mathrm{conv}(K)}\big\|_{0,\mathrm{conv}(K)},
%\leq  \frac{h_K}{\pi}|u|_{2,\mathrm{conv}(K)}.
\end{equation*}
then the Poincar\'e inequality in Lemma 
\ref{lemma:poincare-convex} on the convex set $\mathrm{conv}(K)$, together with the 
fact that $h_K = h_{\mathrm{conv}(K)}$, implies~\eqref{eq:piH1}.
\end{proof}

The approximation result~\eqref{eq:piH1} is usually established using the average 
over the contained disk for a star-shaped element and thus depends on the 
so-called chunkiness parameter in \hyperref[asp:C]{\bf C1} (the diameter over the 
largest radius of the disk with respect to which the domain is star-shaped). Here 
we use the convexity of 
$\mathrm{conv}(K)$ as it does not require any shape regularity of the element $K$.

\begin{lemma}[error estimate for the normal derivative of projection on isotropic 
polygons]
\label{lemma:isopr-nd}
If $K$ satisfies \hyperref[asp:A]{\bf A1--A2} with $0<\gamma\leq 1$, then for any 
$u\in H^2(\mathrm{conv}(K))$, on each $e\subset \partial K$:
\begin{equation}
h_e^{1/2}\norm{\nabla \bigl(u - \Pi  u \bigr)\cdot \bm{n}}_{0,e}
\lesssim \gamma^{-1/2} h_K |u|_{2,\mathrm{conv}(K)}.
\end{equation}
\end{lemma}
\begin{proof}
With \hyperref[asp:A]{\bf A1--A2}, we can apply the weighted trace 
inequality~\eqref{eq:tr-h} in Appendix \hyperref[appendix]{A}, and estimate 
\eqref{eq:piH1} to get
\begin{equation}
\label{eq:u-Piun}
h_e^{1/2}\norm{\nabla \bigl(u - \Pi  u \bigr)\cdot \bm{n}}_{0,e}\lesssim 
{\gamma^{-\frac12}} 
\left(\norm{\nabla \bigl(u - \Pi  u \bigr)}_{0,K} + 
h_e|u|_{2,K} \right)\lesssim 
\gamma^{-\frac12} h_K|u|_{2,\mathrm{conv}(K)}.
\end{equation}
\phantom{!}
%Summing over all $e\subset \partial K$ and using \hyperref[asp:A]{\bf A1}, we get 
%the desired result. 
\end{proof}

Lastly, the broken $1/2$-seminorm term in~\eqref{eq:err-apriori} on each edge which 
again can be easily estimated using a trace inequality as in the traditional 
finite element analysis.

\begin{lemma}[error estimate for the stabilization on isotropic 
polygons]\label{lemma:iso-stab}
If $K$ satisfies \hyperref[asp:A]{\bf A1--A2} with $0<\gamma\leq 1$, then for any 
$u\in H^2(\mathrm{conv}(K))$, 
on each $e\subset \partial K$
\begin{equation}
\label{est:isostab}
|u_I - \Pi  u_I|_{\frac{1}{2}, e} 
\lesssim n_{\mathcal{E}_K} \gamma^{-3/2}\, h_K |u|_{2,\mathrm{conv}(K)}.
\end{equation}
\end{lemma}
\begin{proof}
Split $u_I - \Pi  u_I = (u_I - \Pi  u) + (\Pi  u - \Pi  u_I)$. For the second 
term, one can use the trace inequality~\eqref{eq:tr-1/2} and the interpolation error 
estimate~\eqref{est:isointerp}
\begin{equation}
\left| \Pi  u - \Pi  u_I\right|_{\frac{1}{2}, e}\lesssim
\gamma^{-\frac12} \norm{\nabla (\Pi  u - \Pi  u_I)}_{0,K}\lesssim n_{\mathcal{E}_K} 
\gamma^{-\frac32} \, h_K|u|_{2,K}.
\end{equation}
For the first term $\left| u_I - \Pi u\right|_{\frac{1}{2}, e}$, since $u_I$ and $u$ 
match at the vertices $\bm{a}_e$ and $\bm{b}_e$ of $e$, one has  
\begin{equation}
\label{eq:uIpiue}
(u_I - \Pi  u)\Big|_{\bm{a}_e}^{\bm{b}_e}=(u - \Pi  u)\Big|_{\bm{a}_e}^{\bm{b}_e}
= \int_e \partial_e(u - \Pi  u) \dd s.
\end{equation}
Thus by the Cauchy--Schwarz inequality and because $u_I$ is linear on $e$: 
\begin{equation}
\left| u_I - \Pi u\right|_{\frac{1}{2}, e}
= \left| (u_I - \Pi  u)\big|_{\bm{a}_e}^{\bm{b}_e} 
\right|\leq h_e^{1/2}\norm{\partial_e (u - \Pi  u)}_{0,e}.
\end{equation}
Applying the trace inequality~\eqref{eq:tr-h} and the estimate 
\eqref{eq:piH1} then yields the lemma:
\begin{equation}
\label{eq:uIpiue-est}
h_e^{1/2}\norm{\partial_e (u - \Pi  u)}_{0,e} \lesssim 
\gamma^{-\frac12}\left(\norm{\nabla (u - \Pi  u)}_{0,K} + h_e|u|_{2, K}\right) 
\lesssim \gamma^{-\frac12} h_K|u|_{2,\mathrm{conv}(K)}.
\end{equation}
\end{proof}

We now summarize the convergence result on isotropic meshes as follows.

\begin{theorem}[a priori convergence result for VEM on isotropic 
meshes]\label{theorem:iso}
Assume that $\mathcal{T}_h$ satisfies \hyperref[asp:A]{\bf A1--A2--A3}, 
and that the weak solution $u$ to problem~\eqref{eq:weakform} satisfies the 
regularity result $u\in H^2(\Omega)$; then the following a priori error estimate 
holds for the solution $u_h$ to problem~\eqref{eq:VEMsol} with constant 
dependencies $C_{\gamma}:= 1/\min\{\gamma,1\} $ with $\gamma$ being the 
uniform lower bound for each edge in each $K\in 
\mathcal{T}_h$ satisfying \hyperref[asp:A]{\bf A2}, $C_{\mathcal{E}} := 
\max_{K\in \mathcal{T}_h} 
n_{\mathcal{E}_K}$, and $C_{\omega}:= \max_{K\in \mathcal{T}_h} n_{{\rm 
conv}(K)}$ 
with $n_{\mathrm{conv}(K)}$ in~\eqref{eq:omegaK}:
\begin{equation}\label{eq:isoest}
\tnorm{u-u_h} \lesssim C_{\gamma}^{3/2} C_{\mathcal{E}}^{3/2} 
C_{\omega}^{1/2}\, h \,|u|_{2,\Omega}.
\end{equation}
\end{theorem}
\begin{proof}
First of all, we apply the Stein's extension theorem (\cite[Theorem 6.5]{Stein1970}) 
to $u\in H^2(\Omega)\cap H^1_0(\Omega)$ to obtain $E u \in H^2(\mathbb R^2)$, 
$Eu |_{\Omega} = u|_{\Omega}$, and $| Eu |_{2, \mathbb R^2}\leq C(\Omega) 
| u |_{2,\Omega}$. With this extension $Eu\in H^2\big(\mathrm{conv}(K)\big)$ for any 
$K\in \mathcal T_h$. 

Using Corollary \ref{corollary:errbound}, Lemmas \ref{lemma:iso-interp}, 
\ref{lemma:isopr-nd}, \ref{lemma:iso-stab}, assuming $\mathcal{T}_h$ satisfies 
\hyperref[asp:A]{\bf A2} with $0<\gamma\leq 1$, and the fact that the integral on 
the overlap $\mathrm{conv}(K)\cap \mathrm{conv}(K')$ (when $\partial K \cap \partial 
K'\neq \varnothing$) is repeated $n_{\mathrm{conv}(K)}$ times by 
\hyperref[asp:A3]{\bf A3},  
we obtain 
\begin{equation}
\begin{aligned}
\tnorm{u_I-u_h}^2 
& \lesssim \sum_{K\in \mathcal T_h}
(n_{\mathcal{E}_K}^2 \gamma^{-2} + n_{\mathcal{E}_K}^3 \gamma^{-3} + 
n_{\mathcal{E}_K}^3 \gamma^{-1} )
h_K^2 |Eu|_{2,\mathrm{conv}(K)}^2 
\\
& \lesssim C_{\omega} C_{\mathcal{E}}^3\gamma^{-3} 
h^2|Eu|_{2,\mathrm{conv}(\Omega)}^2 
\lesssim C_{\omega} C_{\mathcal{E}}^3\gamma^{-3}  h^2|u|_{2,\Omega}^2.
\end{aligned}
\end{equation}
By the triangle inequality and Young's inequality, it suffices to estimate 
$\tnorm{u-u_I}$:
\begin{equation}
\tnorm{u-u_I}^2 \lesssim \sum_{K\in \mathcal T_h}\left(
\left\|\nabla \Pi  (u-u_I) \right\|_{0,K}^2 
+ \left| u_I - \Pi u_I \right|_{\frac12, \mathcal{E}_K}^2
+ \left| u - \Pi u \right|_{\frac12, \mathcal{E}_K}^2
\right).
\end{equation}
The first two terms above have been dealt with in Lemmas 
\ref{lemma:iso-interp} and \ref{lemma:iso-stab}. For the third term, by 
\hyperref[asp:A]{\bf A1--A2} and Lemma~\ref{lemma:tr-a}:
\begin{equation}
\label{eq:u-Piuhalf}
\left| u - \Pi u \right|_{\frac12, \mathcal{E}_K}^2 = 
\sum_{e\subset \partial K} \left| u - \Pi u \right|_{\frac{1}{2}, e}^2
\lesssim  n_{\mathcal{E}_K} \gamma^{-1}\norm{\nabla(u - \Pi u)}_{0,K}^2,
\end{equation} 
and the theorem follows from using the projection estimate~\eqref{eq:piH1} on 
$\norm{\nabla(u - \Pi u)}_{0,K}$.
\end{proof}

In light of the approximation error estimates~\eqref{eq:piH1} and~\eqref{eq:isoest}, 
the estimate
\begin{equation}
\left(\sum_{K\in \mathcal{T}_h}\norm{\nabla u - \nabla \Pi  
u_h}_{0,K}^2\right)^{1/2}\lesssim 
h|u|_{2,\Omega}
\end{equation} 
can then be obtained from the triangle inequality. 
Although $u_h$ is not explicitly known, the computable discontinuous piecewise 
linear polynomial $\Pi u_h$ is an optimal order approximation of $u$ in the discrete 
$H^1$-seminorm.

\begin{remark}[other choices of stabilization] 
\label{rmk:stabilizations}
\rm
We remark here that the analyses used in this section apply to other types of 
stabilizations analyzed in 
\cite{Brezzi2013basic,Beirao2017stability,Brenner;Sung:2018VEM,Wang2014wgmixed,Wriggers2016virtual}:
\begin{itemize}
\vspace{3pt}
\item $L^2$-type $S_K(u,v) = \sum_{e\subset \partial K}h_e^{-1}(u, v)_e$,
\vspace{3pt}
\item d.o.f.-type $S_K(u,v) = \sum_{\bm{a}\in 
\mathcal{N}_K}u(\bm{a})v(\bm{a})$,
\vspace{3pt}
\item tangential derivative-type $S_K(u,v) = \gamma_K \sum_{e\subset \partial 
K}(\partial_e u, \partial_e v)_e$.
\end{itemize}
\smallskip

In~\cite{Wang2014wgmixed}, no short edge 
condition is imposed, because an $L^2$-weighted trace inequality is 
used for the $L^2$-type stabilization with an $h_e^{-1}$-weight. In our approach, the 
Poincar\'e-type inequality~\eqref{eq:stab-shortedge} allows short 
edges in an element without further modifying the $h_e^{-1}$ weight. 
Replacing weight $h_e^{-1}$ with $h_K^{-1}$ in these stabilization terms is also 
allowed if an irregular polygon can be embedded into another shape regular one (see 
Section \ref{sec:aniso}). 
$\Box$
\end{remark}

\begin{remark}[removal of the $\log$ factor]\rm 
Comparing with the analyses 
in~\cite{Beirao2017stability,Brenner;Sung:2018VEM} 
by bridging the VEM bilinear form $a_h(\cdot,\cdot)$ with $|\cdot|_{1,K}$ through 
norm equivalences, we opt to work on a weaker norm $\tnorm{\cdot}= 
a^{1/2}_h(\cdot,\cdot)$ to avoid introducing some extra geometric constraints for 
the equivalence between $|\cdot|_{\frac12, \partial K}$ and the stabilization. The 
benefit of this is that the analysis based on the broken $1/2$-seminorm 
$|\cdot|_{\frac12,\mathcal{E}_K}$ does not pay the $\log\big(h_K/\min_{e\subset 
\partial K} h_e\big)$ factor. This log factor 
is unavoidable as well for the d.o.f.-type stabilization since the proof demands 
certain equivalence between $|\cdot|_{\frac12, \partial K}$ and 
$\norm{\cdot}_{\infty,\partial K}$. The 
introduction of $|\cdot|_{\frac12,\mathcal{E}_K}$ evades this problem, as 
Corollary \ref{corollary:normeq-infhalf} further demonstrates that 
$\norm{v_h}_{\infty,\partial K} \eqsim |v_h|_{\frac12,\mathcal{E}_K}$ if $v_h\in 
V_K$ and $\overline{v}_h^{\partial K}=0$.
$\Box$
\end{remark}

The local error analysis in this section is based on \hyperref[asp:A]{\bf 
A1--A2--A3}. In the subsequent section we shall generalize the anisotropic analysis 
for triangles to a special class of anisotropic 
polygons with \hyperref[asp:A]{\bf A2} violated.  

\section{Error analysis for a special class of anisotropic elements}
\label{sec:aniso}
In this section we shall present error analysis for a special class of anisotropic 
elements, which are obtained by cutting a uniform grid consisting of squares of size 
$h$. The present approach could be possibly extended to 
other shape regular meshes; see Remark \ref{rm:shaperegularmesh}.

Henceforth the polygonal mesh can be anisotropic in the sense that it only satisfies 
\hyperref[asp:A]{\bf A1} but there may exist elements not satisfying 
\hyperref[asp:A]{\bf A2}. 
Namely there is a long edge with a short height in such an element, so the 
trace inequalities and Poincar\'e inequalities cannot be applied freely. We shall 
embed an anisotropic element into a shape regular element, e.g. a square of size 
$h$, and apply trace inequalities on the shape regular element. We shall also make 
use of the cancelation of contributions from different boundary edges to bound the 
interpolation error and the stabilization error. 

\subsection{Interpolation error estimates}
\begin{figure}[h]
\includegraphics[width=\textwidth]{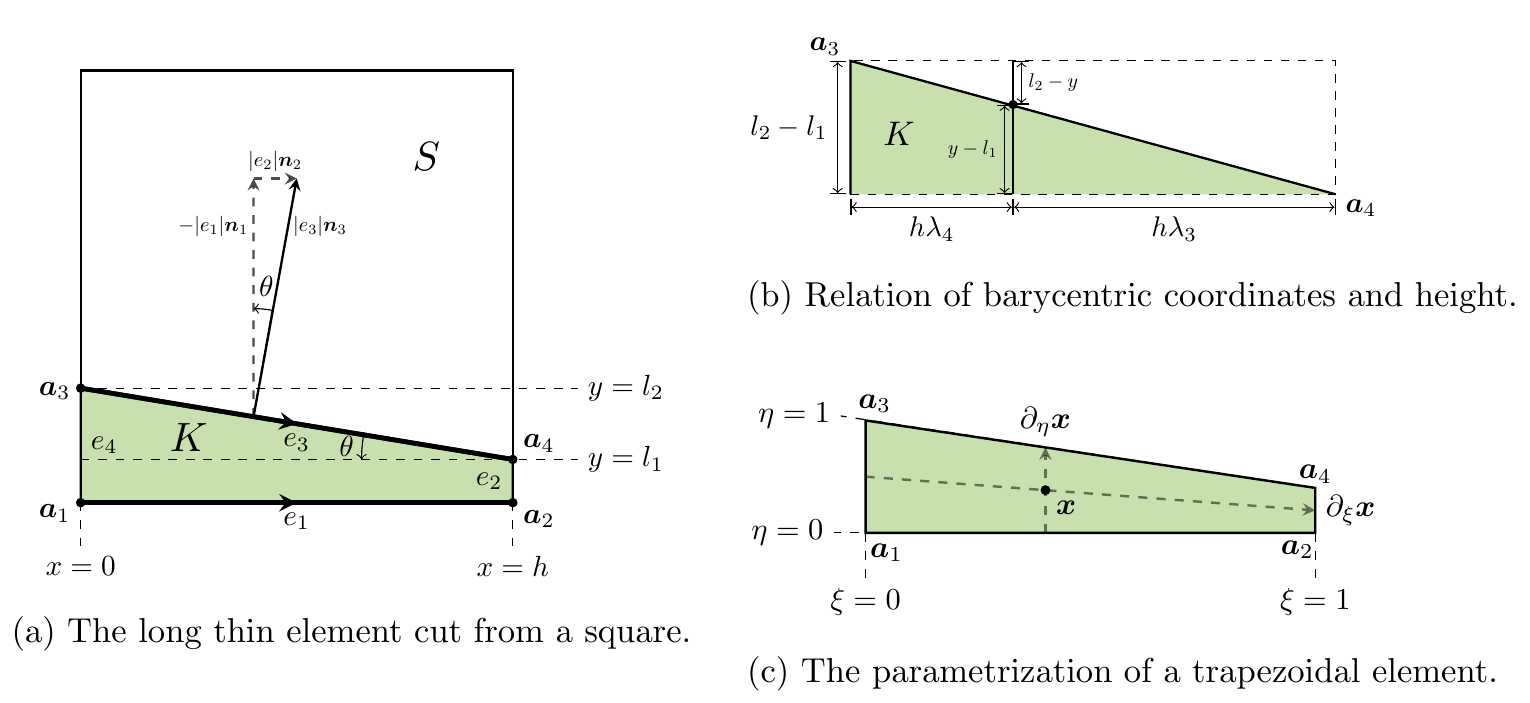}
\caption{The anisotropic element setting used in Lemma 
\ref{lemma:aniso-interp}, \ref{lemma:aniso-projnd}, 
\ref{lemma:aniso-Piu-uIedge}, and \ref{lemma:poincare-e}. }
\label{fig:anielem}
\end{figure}
\begin{lemma}[a projection-type error estimate for the 
interpolation in a cut element]
\label{lemma:aniso-interp}
Assume $K$ is obtained by cutting a square $S$; then
\begin{equation}
\label{est:interp}
\norm{\nabla \Pi  (u - u_I)}_{0,K} \lesssim h_K |u|_{2,S}, \qquad u\in H^2(S).
\end{equation}
\end{lemma}
\begin{proof}
We process as before (cf.~\eqref{est:isointerp}), by letting $q = \Pi  (u 
- u_I)$ to obtain 
\begin{equation}
\label{eq:u-uIonbd}
\norm{\nabla \Pi  (u - u_I)}_{0,K} = \sum_{e\subset \partial K} (u - u_I, 
\nabla q \cdot \bs{n})_e.
\end{equation}
If $K$ is a right triangle, the error analysis can be applied for 
triangles satisfying the maximum angle condition~\cite{Babuska.I;Aziz.A1976}. If $K$ 
is a pentagon, $K$ is isotropic by Lemma \ref{lemma:SKR} and thus the estimate on 
isotropic polygons can be applied. So here $K$ is assumed to a right trapezoid, of 
which one side is aligned with $S$ (see Figure \ref{fig:anielem}(a)), and $l_2>l_1$ 
and $l_1$ could vanish in which case the right trapezoid is degenerated to a 
triangle. For edges $e_2$ and $e_4$, \hyperref[asp:A]{\bf A1--A2} holds 
and thus proofs in Lemma \ref{lemma:iso-interp} can be applied, while for edge 
$e_1$ and $e_3$, $l_e\ll h_e$. 

By 
parametrizing using $0\leq x\leq h$, the edge $e_1$ from $\bs{a}_1$ to $\bs{a}_2$: 
$\bs{r}_1 = \langle x,0\rangle$, and $e_3$ from $\bs{a}_3$ to $\bs{a}_4$: 
$\bs{r}_3 = \langle x, l_1 - x\tan \theta\rangle$ 
(where $\cos\theta = -\bs{n}_1\cdot \bs{n}_3$), and decomposing the outer unit 
normal to $e_3 $ as $\bs{n}_3 = -\bs{n}_1 \cos\theta + \bs{n}_2 \sin\theta$, the 
boundary integrals on $e_1$ and $e_3$ can be written as
\begin{equation}\label{eq:u-uIdq}
\begin{aligned}
&\;(u - u_I, \nabla q \cdot \bs{n})_{e_1} + 
(u - u_I, \nabla q \cdot \bs{n})_{e_3}
\\
= & \underbrace{\int_0^h \left[(u-u_I)(\bs{r}_1) -(u-u_I)(\bs{r}_3) \right] \nabla 
q\cdot \bs{n}_1 \dx}_{(\flat)} + 
\underbrace{\int_0^h (u-u_I)(\bs{r}_3) \nabla q\cdot \bs{n}_2\tan\theta 
\dx}_{(\natural)}.
\end{aligned}
\end{equation}
Before moving into the details, we use a special case: $K$ is a rectangle with side 
$h$ and height $l$ to explain briefly the cancelation technique. In this case, $\bs 
n_3 = -\bs n_1$ and thus the term $(\natural)$ disappears. The interpolation error 
$u-u_I$ along the horizontal edges $\bs e_1$ and $\bs e_3$ is of order $h$ and the 
difference $(u-u_I)(\bs{r}_1) -(u-u_I)(\bs{r}_3)$ is of order $l$. So a refined 
scale $l h$ instead of $h^2$ is obtained.

For a general trapezoid, $\bs e_3$ may not be parallel to $\bs e_1$. But the factor $h\tan \theta = l_2 - l_1$ weighing like a short edge.
The integral $(\natural)$ above can be dealt with using a similar argument with 
\eqref{eq:intp-1}:
\begin{equation}
\label{eq:int-natural}
|(\natural)|
\lesssim h(l_2-l_1) |u|_{\frac{3}{2},e_3} |\nabla q| 
\lesssim |K|^{1/2}|u|_{2,S} \norm{\nabla q}_{0,K}.
\end{equation}

The integral $(\flat)$ in~\eqref{eq:u-uIdq} is now the focus. For a point $ \bs r_1 = \langle x, 0\rangle \in e_1$,  the barycentric coordinate $\lambda_i(\bs r_1)$ ($i=1,2$) are positive numbers satisfying $\langle x, 0\rangle = \lambda_1(x,0) \bs a_1 +  \lambda_2(x,0) \bs a_2$.
Using the fundamental 
theorem of calculus, it is straightforward to arrive at the following error 
representation on $e_1$:
\begin{equation}\label{eq:e1}
(u-u_I)(\bs{r}_1) = \lambda_1(x,0)\int^x_0 \partial_x u(t,0)\dd t
- \lambda_2(x,0)\int^h_x \partial_x u(t,0)\dd t.
\end{equation}
Similarly on $e_3$, for a point $\bs{r}_3 = \langle x,y\rangle \in e_3$, introduce the barycentric coordinate $\lambda_i(\bs r_3)$ ($i=3,4$) satisfying $\langle x, y\rangle = \lambda_3(x,y) \bs a_3 +  \lambda_4(x,y) \bs a_4$. We have the error representation on $e_3$:
\begin{equation}
(u-u_I)(\bs{r}_3) = \lambda_3(\bs r_3)\int_{\bs a_3}^{\bs r_3} \partial_t u \dd t
- \lambda_4(\bs r_3)\int_{\bs r_3}^{\bs a_4} \partial_t u\dd t.
\end{equation}
To be able to cancel with terms in~\eqref{eq:e1}, we further decompose the line 
integral along $\bs e_3$ into coordinate directions:  
\begin{equation}
\begin{aligned}
(u-u_I)(\bs{r}_3) 
& = -\lambda_3(x,y)\int^{l_2}_y \partial_y u(x,\tau)\dd\tau
+ \lambda_3(x,y)\int^x_0 \partial_x u(t,l_2)\dd t
\\
&\quad + \lambda_4(x,y)\int^{y}_{l_1} \partial_y u(x,\tau)\dd\tau
- \lambda_4(x,y)\int^h_x \partial_x u(t,l_1)\dd t.
\end{aligned}
\end{equation}
It follows from the fact that $\lambda_1(x,0) = \lambda_3(x,y)$ and $\lambda_2(x,0) 
= \lambda_4(x,y)$:
\begin{equation}
\begin{aligned}
& (u-u_I)(\bs{r}_1) -(u-u_I)(\bs{r}_3)
\\
= & \; \underbrace{\lambda_1(x,0)\int^x_0 \big[\partial_x u(t,0) - \partial_x 
u(t,l_2) \big]\dd t
+ \lambda_2(x,0)\int^h_x \big[\partial_x u(t,l_1) - \partial_x u(t,0) 
\big]\dd t}_{(\dagger)}
\\
& + \underbrace{\lambda_3(x,y)\int^{l_2}_y \partial_y u(x,\tau)\dd\tau
- \lambda_4(x,y)\int^y_{l_1} \partial_y u(x,\tau)\dd\tau}_{(\ddagger)}.
\end{aligned}
\end{equation}
Assuming $l_1\neq l_2$, otherwise $(\ddagger)=0$, applying the mean value theorem on 
both integrals in $(\ddagger)$, there exists $\xi_2 \in (y,l_2)$, and $\xi_1\in 
(l_1,y)$ such that
\begin{equation}
\label{eq:second}
\begin{aligned}
(\ddagger) &= \lambda_3(x,y) (l_2-y) \partial_y u(x,\xi_2) 
- \lambda_4(x,y) (y-l_1) \partial_y u(x,\xi_1)\\
&
= \frac{(l_2-y)(y-l_1)}{l_2-l_1} 
 \int^{\xi_2}_{\xi_1} \partial_{yy} u(x,\tau)\dd \tau
\leq \frac{l_2-l_1}{4}
\int^{l_2}_{l_1} |\partial_{yy} u(x,y)|\dd y .
\end{aligned}
\end{equation}
wherein the second step the geometric meaning of barycentric coordinates 
(see Figure \ref{fig:anielem}(b)) is used:
\begin{equation}
\frac{\lambda_3(x,y)}{y-l_1} = \frac{\lambda_4(x,y)}{l_2-y} = \frac{1}{l_2-l_1}.
\end{equation}
For $(\dagger)$, an estimate can be obtained as follows:
\begin{equation}\label{eq:first}
\begin{aligned}
\bigl|(\dagger) \bigr| & \leq 
{\left\vert 
\lambda_1(x,0)\int^x_0 \int^{l_2}_0\partial_{xy} u(t,\tau)\dd\tau dt 
\right\vert} 
+  
{\left\vert 
\lambda_2(x,0)\int^h_x \int^{l_1}_0\partial_{xy} u(t,\tau)\dd\tau  dt
\right\vert} 
\\
& \leq \norm{\partial_{xy} u}_{0,S} h^{1/2}l_1^{1/2}
+\norm{\partial_{xy} u}_{0,S} h^{1/2}l_2^{1/2} \lesssim |K|^{1/2} |u|_{2,S}.
\end{aligned}
\end{equation}
By the estimate in~\eqref{eq:second}, together with~\eqref{eq:first}, the integral 
$(\flat)$ in~\eqref{eq:u-uIdq} can be estimated as follows:
\begin{equation}
\label{eq:int-flat}
\begin{aligned}
|(\flat)| \leq  & \; \int^h_0 |(\dagger)| \, |\nabla q|\,\dx 
+ \frac{l_2-l_1}{4}\int^h_0 
\int^{l_2}_{l_1} |\partial_{yy} u(x,y)|\,|\nabla q|\dd y \dx
\\ 
\lesssim & 
\; \; h |K|^{1/2} |\nabla q|\,|u|_{2,S}  + 
\frac{l_2-l_1}{4} \left(\int_K |\partial_{yy} u(x,y)|^{2} \dx \dd 
y\right)^{\frac{1}{2}}
|K|^{1/2} |\nabla q|,
\end{aligned}
\end{equation}
which can be bounded by $h_K|u|_{2,S} \norm{\nabla q}_{0,K}$. In summary, we have 
proved the following inequality and~\eqref{est:interp} is obtained by canceling 
$\norm{\nabla q}_{0,K}$ on both sides:
\begin{equation}
\norm{\nabla \Pi  (u - u_I)}_{0,K}^2 \lesssim h_K|u|_{2,S} \norm{\nabla q}_{0,K}.
\end{equation}
\end{proof}

\begin{remark}[generalization to certain anisotropic 
polygons]\label{rm:generalization}
\rm
The proof of the interpolation estimate is not restricted to anisotropic 
quadrilaterals. 
Heuristically speaking, if there is a long edge $e$ (which can be $e_1$ in 
Figure \ref{fig:anielem}(a)) that supports a 
short height $l_e$ toward the interior of an anisotropic element $K$. Meanwhile 
if there exists another long edge ($e_3$ in the aforementioned case) on $\partial K$ 
pairing with this long edge, then the cancelation will occur for 
terms $(u - u_I, \nabla q \cdot \bs{n})_{e_i}$ ($i=1,3$) among the ones in 
\eqref{eq:u-uIonbd}. Even on quadrilaterals, our analysis on VEM relaxes the 
stringent constraints imposed on isoparametric elements for anistropic 
quadrilaterals (cf.~\cite{Acosta2000error} and the references therein). However, 
a precise characterization of the class of anisotropic elements for which this 
analysis can be applied seems difficult. For example, the same cancelation argument 
can be generalized to the element in Figure \ref{fig:starelem}(b), but not to 
the element in Figure \ref{fig:starelem}(a). The existence of three 
long edges in Figure \ref{fig:starelem}(a) forbids a straightforward pairing and 
cancelation trick to yield the desirable estimate.
$\Box$
\end{remark} 

\begin{figure}[h]
\centering
\includegraphics{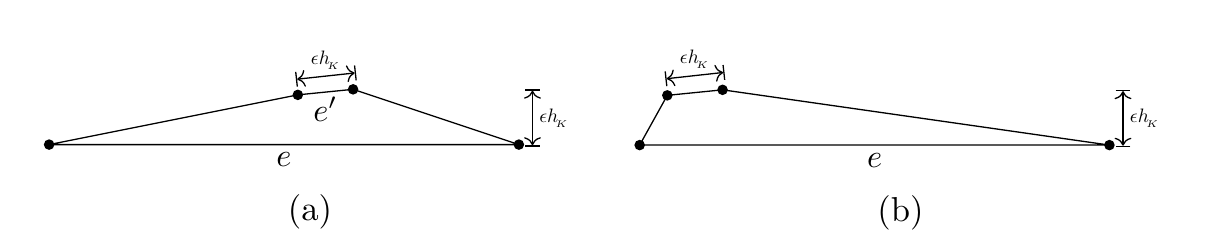}
\caption{The examples of highly degenerate quadrilaterals while $\epsilon=o(h_K)$ as 
$h_K\to 0$. }
\label{fig:aniso-interpolation}
\end{figure}

\subsection{Normal derivative of the projection error}
For the term $\nabla \bigl(u - \Pi  u \bigr)\cdot \bm{n}$, only the anisotropic 
quadrilateral needs attention as if $K$ is a triangle cut from a square, there is no 
stabilization due to $(\operatorname{I}-\Pi)V_K=\varnothing$. As we have mentioned 
earlier, the trace inequalities are applied towards a larger and shape regular 
element. Then a refined Poincar\'e inequality in the following lemma is needed to 
estimate the approximation on this extended element. The reason is that one cannot 
apply the average-type Poincar\'e inequality directly over a subdomain (cf. Lemma 
\ref{lemma:poincare-subset}), with $\omega = |K|$ as $|K|\ll h_S^2$ when $K$ is 
anisotropic and thus $h_S/|K|^{1/2}\gg 1$. 

\begin{lemma}[Poincar\'{e} inequality on an anisotropic cut element]
\label{lemma:poincare-aniso}
Let $S$ be a square and $K\subset S$ is a quadrilateral from cutting $S$ 
by a straight line; then
\begin{equation}
\label{eq:poincare-aniso}
\norm{v - \overline{v}^K}_{0,S} \lesssim h_{S} \norm{\nabla v}_{0,S}, \quad v\in 
H^1(S),
\end{equation}
holds with a constant independent of the ratio $|K|/|S|$.
\end{lemma}
\begin{proof}
The illustration of a possible configuration of $S$ versus $K$ in Figure 
\ref{fig:anielem} is used. We shall use the average on $e = e_1$ as a bridge:
$$
\norm{v - \overline{v}^K}_{0,S}\leq \norm{v - \overline{v}^{e}}_{0,S} + 
\norm{\overline{v}^K - \overline{v}^e}_{0,S}. 
$$
The first term can be estimated using the Poincar\'e inequality with average zero on 
an edge (see~\eqref{eq:poincare-eR} in Lemma \ref{lemma:poincare-e}):
$$
\norm{v - \overline{v}^{e}}_{0,S} \lesssim h_S\|\nabla v\|_{0,S}.
$$
For the second term, by the definition of $\overline{v}^K$ and Cauchy--Schwarz 
inequality,
\begin{align*}
\norm{\overline{v}^K - \overline{v}^e}_{0,S} \leq {|S|^{1/2}}{|K|^{-1/2}}\| v - 
\overline{v}^e \|_{0,K}.
\end{align*}
With a slight notation change of the proof of~\eqref{eq:poincare-eA}, we have
$$
\norm{v - \overline{v}^e}_{0,K} \lesssim (l_e h_e)^{1/2} |v|_{\frac12,e} + l_e 
\norm{\nabla v}_{0,K}
$$
with $l_e = (l_1+l_2)/2$. Then the desired inequality follows from applying the fact 
$|K|\geq l_eh_e/2$ and the trace inequality~\eqref{eq:tr-l2} on $S$:
\begin{align}
\norm{\overline{v}^K - \overline{v}^e}_{0,S} \lesssim h_S \left (|v|_{\frac12,e} + 
\|\nabla v\|_{0,K} \right )\lesssim h_S\| \nabla v\|_{0,S}.
\end{align}
\end{proof}

\begin{lemma}[error estimate for the normal derivative of projection on anisotropic 
cut elements]
\label{lemma:aniso-projnd}
Let $K$ be an anisotropic quadrilateral cut from a square 
$S$, with at least one edge $e\subset \partial K$ satisfying \hyperref[asp:A]{\bf 
A1} only, not \hyperref[asp:A]{\bf A2}. Then for all $e\subset \partial K$
\begin{equation}
h_e^{1/2}\norm{\nabla \bigl(u - \Pi  u \bigr)\cdot \bm{n}}_{0,e}
\lesssim h_K |u|_{2,S}, \qquad u\in H^2(S).
\end{equation}
\end{lemma}
\begin{proof}
For edges satisfying \hyperref[asp:A]{\bf A1-A2}, one can use the trace 
estimate~\eqref{eq:tr-h} with the weight $h_e$ the same way with the estimate 
\eqref{eq:u-Piun} in Lemma \ref{lemma:isopr-nd}. For a long 
edge not satisfying \hyperref[asp:A]{\bf A2}, i.e., $l_e\ll h_e \eqsim h_K$, there 
are two cases. If $e\cap \partial S\neq \varnothing $ (e.g., $e_1$ in Figure 
\ref{fig:anielem}(a)), the trace estimate~\eqref{eq:tr-h} can be applied treating 
$e$ as a boundary edge to $S$ with the weight $h_S$.
Otherwise, $e= \partial K \cap \partial R$ is the 
cut line segment (e.g., $e_3$ in Figure \ref{fig:anielem}(a)), where $R := 
S\backslash (K\cup e)$. By Lemma~\ref{lemma:SKR}, $R$ is isotropic, to which the 
trace inequality~\eqref{eq:tr-l2} can be applied with the weight $h_R\eqsim h_S$.
In both cases, we can get
\begin{equation}
\label{eq:u-piu-est}
\norm{\nabla \bigl(u - \Pi  u \bigr)\cdot \bm{n}}_{0,e} 
\lesssim h_S^{-1/2}\norm{\nabla \bigl(u - \Pi  u \bigr)}_{0,S} + h_S^{1/2} |u|_{2,S}.
\end{equation}
Lastly, since $\nabla \Pi u = \nabla \Pi^{\nabla}_K u =\overline{\nabla u}^{K}$, 
applying the Poincar\'{e} inequality on a cut element in Lemma 
\ref{lemma:poincare-aniso} to the term 
$\norm{\nabla \bigl(u - \Pi  u \bigr)}_{0,S}$ yields the desired result.
\end{proof}

\subsection{Stabilization term}
In an anisotropic element $K$, the stabilization error in~\eqref{eq:err-apriori} is 
again split as $u_I - \Pi  u_I = (u_I - \Pi  u)  + (\Pi  u  - \Pi  u_I)$. 
Using the representation~\eqref{eq:uIpiue}, the following similar estimate 
can be proved for the first term on $e\subset \partial K$,
\begin{equation}
\label{eq:uiPiuedge}
|u_I - \Pi  u|_{\frac{1}{2}, e} \leq h_e^{1/2}\norm{\partial_e (u - \Pi  
u)}_{0,e}\lesssim h_S |u|_{2,S}, \qquad u\in H^2(S).
\end{equation}
When applying the weighted trace inequality~\eqref{eq:tr-l2} following 
\eqref{eq:uIpiue-est}, the 
argument is the same with the one in Lemma \ref{lemma:iso-stab} if $e$ is 
a short edge. If $e$ is a long edge, the difference is that we treat $e$ as one 
boundary edge not to $K$, but to the square $S$ or $K$'s neighboring isotropic element, 
as the trick used in proving~\eqref{eq:u-piu-est}. The 
second term $\Pi  u  - \Pi u_I$ is the focus of the subsequent lemma. 
\begin{lemma}
\label{lemma:aniso-Piu-uIedge}
Let $K$ be an anisotropic quadrilateral cut from a square $S$, and $K$ satisfies 
\hyperref[asp:A]{\bf A1} only, then it holds that on any 
$e\subset \partial K$
\begin{equation}\label{eq:PiuIedge}
\snorm{\Pi  (u -u_I)}_{\frac12,e} \lesssim h_K|u|_{2,S},\qquad u\in H^2(S).
\end{equation}
\end{lemma}
\begin{proof}
First $\big|{\Pi  (u  - u_I) }\big|_{\frac12,e} = \left| \Pi  (u  - u_I) 
|_{\bm{a}_e}^{\bm{b}_e} \right|$. Since $\nabla \Pi  u = \overline{\nabla u}^K$ 
and $\Pi  u \in \mathbb P_1(K)$, $\Pi  u = (\bs x - \bm{c}_1)\cdot 
\overline{\nabla u}^K + c_0$. By taking the difference at two end points of edge 
$e$, the following representation holds with $\bs t_e$ being the unit tangential 
vector of edge $e$ pointing from $\bm{a}_e$ to $\bm{b}_e$:
\begin{align}
\label{eq:piu-uIhalf}
\Pi  (u -  u_I) \Big|_{\bm{a}_e}^{\bm{b}_e} = h_e\bs t_e 
\cdot \frac{1}{|K|}\int_K \nabla (u-u_I) = \frac{h_e}{|K|} 
\underbrace{\sum_{e'\subset \partial 
K}\int_{e'} (u-u_I) \bs t_e\cdot \bs n_{e'} \dd s}_{(\sharp)}.
\end{align}

{\emph{Case} 1.} Taking the notations in Figure \ref{fig:anielem}(a), we 
first consider a short edge $e = e_2$ or $e_4$. Then $\bs t_e \cdot \bs n_{e'}= 0 $ 
for $e' = e_2$ or $e_4$. Set $q = -y$ and notice that $\norm{\nabla q}_{0,K} 
=|K|^{1/2}$. Now $(\sharp)$ 
bears the same form with~\eqref{eq:u-uIdq}:
\begin{equation}
(\sharp) = (u - u_I, \nabla q \cdot \bs{n})_{e_1} + 
(u - u_I, \nabla q \cdot \bs{n})_{e_3},
\end{equation}
The representation~\eqref{eq:piu-uIhalf} of 
$\big|{\Pi  (u  - u_I) }\big|_{\frac12,e} $ can be then handled by the established 
estimates in~\eqref{eq:int-natural} and~\eqref{eq:int-flat} which imply 
$|(\sharp)|\lesssim  h_K |u|_{2,S} \norm{\nabla q}_{0,K}$: 
\begin{equation}
\label{eq:piu-uIe} 
\big|{\Pi  (u  - u_I) }\big|_{\frac12,e}= \frac{h_e}{|K|} |(\sharp)|\lesssim 
\frac{h_e\,h_K}{|K|^{1/2}} |u|_{2,S} \lesssim h_K|u|_{2,S}.
\end{equation}

{\emph{Case} 2.}
If $e = e_1$ is a long edge (and $e_3$ can be treated similarly), then $(\sharp)$ in 
\eqref{eq:piu-uIhalf} is:
\begin{equation}
\label{eq:piu-uIe1}
(\sharp) = (u - u_I, \nabla q \cdot \bs{n})_{e_2\cup e_4} 
+ \int_{e_3} (u - u_I)\sin\theta\dd s,
\end{equation}
where $\bm t_{e_1} = \nabla q$, and $q=x$ with $\norm{\nabla q}_{0,K} =|K|^{1/2}$. 
The first integral in~\eqref{eq:piu-uIe1} is on two short edges, where the isotropic 
estimate~\eqref{eq:intp-1} can be applied: 
\begin{equation} 
\label{eq:piu-uIe2e4}
(u - u_I, \nabla q \cdot \bs{n})_{e_2\cup e_4} \lesssim 
h_{e_4} |u|_{2,K} \norm{\nabla q}_{0,K} = h_{e_4} |K|^{1/2} |u|_{2,K}.
\end{equation}
The second integral in~\eqref{eq:piu-uIe1} can be estimated using 
that small $\sin\theta$ factor:
\begin{equation}
\int_{e_3} (u - u_I)\sin\theta\dd s \leq h_{e_3}^{1/2}\sin\theta \norm{u-u_I}_{0,e_3}
\lesssim h_{e_3}^{2}\sin\theta\, |u|_{\frac{3}{2},e_3}.
\end{equation}
For $|u|_{\frac{3}{2},e_3}$, applying the trace 
inequality in~\eqref{lemma:tr-a} for $\nabla u$ on $e_3$ toward $S$, 
meanwhile combining the first identity in~\eqref{eq:piu-uIe}, 
\eqref{eq:piu-uIe2e4}, and $h_{e_3}\sin\theta = l_2-l_1$, yields
\begin{eqnarray}
\snorm{\Pi  (u  - u_I) }_{\frac12,e_1}\lesssim \frac{h_{e_1}}{|K|} 
\left(h_{e_3}(l_2-l_1) + h_{e_4} |K|^{1/2}\right) 
|u|_{2,S}\lesssim h_K |u|_{2,S}.
\end{eqnarray}
\end{proof}

\begin{remark}\rm 
The broken $1/2$-seminorm of a linear polynomial is equivalent to the difference of 
it on the two end points of an edge, which further gives a tangential vector of the 
edge. Then its inner product with the normal vectors leads to scales in different 
directions. If $L^2$- or $\ell^2$-type stabilization is used, which involves the sum 
not the difference of function values at two end points, such cancelation is not 
possible on anisotropic elements.
$\Box$
\end{remark}

\subsection{Convergence results}
We summarize the following convergence result. 
\begin{theorem}[an \textit{a priori} convergence result for VEM on a special class 
of anisotropic meshes]
\label{theorem:aniso}
Assume that $\forall K\in \mathcal{T}_h$ either satisfies \hyperref[asp:A]{\bf 
A1--A2--A3}, or is cut from a square of size $h$ by one straight line. Under the 
same setting with Theorem \ref{theorem:iso} 
for $u$ and $u_h$, it holds that:
\begin{equation}\label{eq:anisoconvergence} 
\tnorm{u-u_h} \lesssim  h \,| u |_{2,\Omega}.
\end{equation}
\end{theorem}
\begin{proof}

The proof, which instead uses estimate~\eqref{eq:uiPiuedge}, Lemmas 
\ref{lemma:aniso-interp}, \ref{lemma:aniso-projnd}, and 
\ref{lemma:aniso-Piu-uIedge}, 
is almost identical to that of Theorem \ref{theorem:iso} with $C_{\omega} = 2$, 
$C_{\mathcal{E}} = 5$. Except that when using 
the trace inequalities to prove the estimate~\eqref{eq:u-Piuhalf}, the trace is 
lifted toward the square $S$ not $K$, and the refined Poincar\'e inequality in 
Lemma \ref{lemma:poincare-aniso} is applied on $\norm{\nabla(u - \Pi u)}_{0,S}$. 
\end{proof}
\begin{remark}[general shape regular background meshes]
\label{rm:shaperegularmesh}\rm 
The present approach based on cutting could be possibly extended to 
other shape regular background meshes, which might be more suitable for  domain 
$\Omega$ with complex boundary. A straight line will cut a shape regular triangle 
into a triangle satisfying the maximum angle condition and a possible anisotropic 
trapezoid, which can be mapped to the one in Figure \ref{fig:anielem}(a) with a 
bounded Jacobian. $\Box$
\end{remark}
\begin{remark}[a simple polygonal mesh generator]\rm
The convergence result~\eqref{eq:anisoconvergence} justifies 
a simple mesh generator for a 2-D domain $\Omega$. First a uniform partition 
$\mathcal{T}_h$ is used to enclose $\Omega$, then near $\partial \Omega$, one shall 
find the cut and keep the polygons inside 
$\Omega$ in $\mathcal{T}_h$. VEM based on 
this mesh delivers an optimal first order approximation if $\partial \Omega$ is 
well-resolved by $\mathcal{T}_h$. Our analysis offers a theoretical justification of 
VEM convergence on partial-conforming polygonal mesh cut from a regular 
triangulation~\cite{Benedetto;Berrone;Pieraccini;Scialo:2014virtual}, as well as an 
alternative perspective on methods like fictitious domain FEM (see 
\cite{Burman;Hansbo:2010Fictitious} and the references therein).
$\Box$
\end{remark}

\begin{remark}[restriction of the current approach]\label{rm:restriction}\rm
 Our anisotropic error analysis is limited to the cut of a shape regular mesh by one 
 straight line. If a square is cut into several thin slabs the aspect 
ratio of the obtained elements impacts the estimate due to the repeated use 
of the trace lifting to the full square. While such a case can be dealt with in the 
traditional finite element analysis, the reason is that it suffices to prove only 
the anisotropic interpolation error estimate, and unlike in VEM, no stabilization 
terms on the edges are needed.
\end{remark}
\section{Conclusion and Future Work}
In this paper, we have introduced a new set of 
localized geometric assumptions for each edge of an element and established a new 
mechanism to show the convergence of the lowest order (linear) VEM under a weaker 
discrete norm. The new analysis enabled us to handle short edge naturally, and 
explained the robustness of 
VEM on certain anisotropic element cut from a square element. 

Extending the analysis to three-dimensional (3-D) \cite{Cao;Chen;Lin:2018VEM} and 
nonconforming VEM \cite{Cao;Chen:2018AnisotropicNC} is our ongoing 
study. In 3-D, \cite{Beiraodaveiga2017high} already has shown some 
promising numerics, yet on each face of the element, the broken $1/2$-seminorm no 
longer enjoys a simple formula as the one in~\eqref{eq:difference-half}. How to find 
an easily computable alternative stabilization is not a 
trivial task, providing the error estimate can be theoretically proved immune to 
unfavorable scaling due to anisotropy. For nonconforming VEM, since the 
interpolation is no longer a polynomial any more on boundary of each element by the 
standard construction in~\cite{Ayuso2015nc}, the estimates in section \ref{sec:iso} 
and \ref{sec:aniso} needs to be modified.

The mesh conditions and error analysis on isotropic elements in this paper can be 
easily generalized to 3-D but not the anisotropic error analysis. Re-examination of 
the anisotropic error analysis in 3-D and for other elements on 
simplices~\cite{AlShenk1994,Duran.R1999,Krizek1991semiregular,Krizek1992max} is 
needed.

\section*{Appendix A: Trace inequalities}
\label{appendix}
When using a trace inequality, one should be extremely careful as the constant  
depends on the shape of the domain. In this appendix, we shall re-examine 
several trace inequalities with more explicit analyses on the geometric conditions. 

\begin{lemma}[a trace inequality on the reference triangle; Chapter 2 Lemma 5.2 in 
\cite{Necas1967}]
\label{lemma:tr-r}
Let $\widehat{T}$ be the triangle with vertices $(0,0)$, $(1,0)$, and $(1,1)$, and 
$e$ be the edge from $(0,0)$ to $(1,1)$, then for any $v\in H^1(\widehat{T})$, it 
holds that:
\begin{equation}
|v|_{\frac12,e}  \leq  2 \norm{\nabla v}_{0,\widehat{T}}.
\end{equation}
\end{lemma}

When using the scaling argument from the reference triangle above, it depends on the 
Jacobi matrix which in turn depends on the ratio of $l_e$ and $h_e$.

\begin{lemma}[a trace inequality of an extension type for an edge in a polygon]
\label{lemma:tr-a}
If on a polygon $K$, an edge $e\subset\partial K$ has the height $l_e$ 
defined in~\eqref{eq:local-le}, and $T_e$ is the triangle with base $e$ and height 
$l_e$, then for any $v\in H^1(K)$, it holds that:
\begin{equation}\label{eq:trlh}
|v|_{\frac12,e} \lesssim \left( \frac{h_e}{l_e} + 
\frac{l_e}{h_e}\right)^{1/2} \norm{\nabla v }_{0,K}.
\end{equation}
Furthermore, if the edge $e$ satisfies \hyperref[asp:A]{\bf A2} with constant 
$0<\gamma\leq 1$:
\begin{equation}\label{eq:tr-1/2}
|v|_{\frac12,e} \lesssim \gamma^{-1/2}\, \norm{\nabla v }_{0,K}.
\end{equation}
\end{lemma}
\begin{proof}
Let the triangle with $e$ as a base be $T_e$. 
\begin{figure}[!b]
\begin{center}
\includegraphics[width=0.7\textwidth]{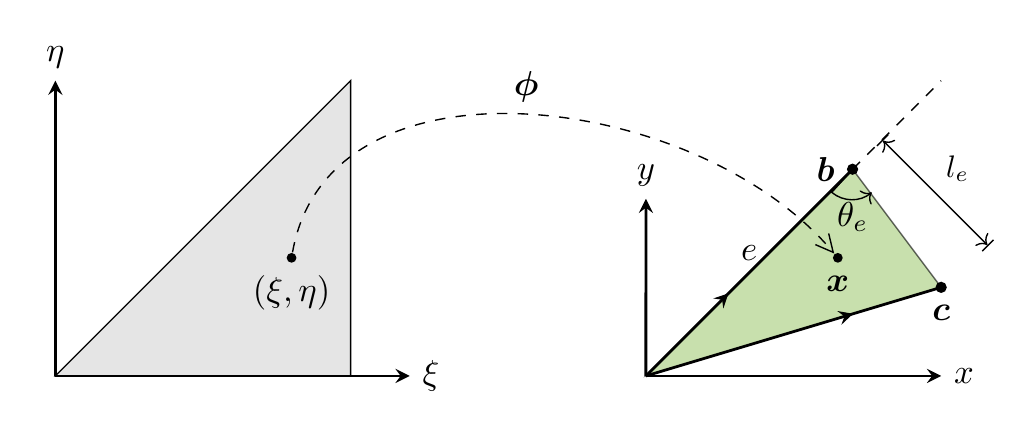}
\caption{Mapping $\phi$ from the reference triangle $\hat T$ to $T_e$.}
\label{fig:ref}
\end{center}
\end{figure}
There 
exists an affine mapping 
$\bm{\phi}: \widehat{T} \to T_e\subset K$, where $\widehat{T}$ is the 
unit triangle in Lemma \ref{lemma:tr-r}. Without loss of generality, it is assumed 
that the edge of interest $e$ aligns with $y=x$ in 
$\mathbb{R}^2$, which shares the same tangential direction as the hypotenuse of 
$\widehat T$. One vertex of the edge $e$ is assumed to be the 
origin (see Figure~\ref{fig:ref}), then:
\begin{equation}
\label{eq:tr-j}
\bm{x} = \bm{\phi}(\xi,\eta) = \xi \bm{c} + \eta (\bm{b} - \bm{c}), \text{ for } 
0\leq \eta \leq \xi \leq 1,
\end{equation}
where $\bm{b}=h_e\langle 1,1 \rangle$ and $\bm{c}$ represent the other two 
vertices. Using the parametrization~\eqref{eq:tr-j}: 
$e\ni \bm{x} = \bm{\phi}(\xi,\xi)= h_e \langle \xi,\xi\rangle$ for $0\leq 
\xi\leq 1$, thus $\dd s(\bm{x}) = \sqrt{2}\,h_e \dd\xi$, and the $1/2$-seminorm defined 
in~\eqref{eq:norm-half} is:
\begin{equation}
|v|_{\frac12,e}^2 = \int_{(0,1)^2} \frac{|v\big(\bm{\phi}(t,t)\big) - 
v\big(\bm{\phi}(\tau,\tau)\big)|^2}{2h_e^2|t-\tau|^{2}} 
\,2h_e^2 \,\dd t \dd \tau = |\widehat{v}|_{\frac12,\hat{e}}.
\end{equation}
By Lemma \ref{lemma:tr-r}, let 
the $ \widehat{\nabla}$ denote the gradient in $(\xi,\eta)$ on $\widehat{T}$,
\[
|v|_{\frac12,e}^2 \leq 4 
\norm{ \widehat{\nabla} v \big(\bm{\phi}(\xi,\eta)\big)}_{0,\hat T}^2 \leq  \; 
4\int_{\hat T} \left(|\nabla v|^2
\left| \partial_{\xi}\bm{\phi} \right|^2
+  |\nabla v|^2 \left| \partial_{\eta}\bm{\phi}\right|^2 \right)\, \dd \eta \dd \xi.
\]
One has $|J_{\bm{\phi}}| = 2|T_e| = h_e \,l_e$ for the mapping~\eqref{eq:tr-j}, and 
$\left| \partial_{\xi}\bm{\phi} \right|^2 = |\bs c|^2$ 
and $\left| \partial_{\eta}\bm{\phi} \right|^2 = |\bs b - \bs c|^2$. Inserting these 
identities into above integral yields:
\begin{equation}
\label{eq:tr-constant}
|v|_{\frac12,e}^2  \leq 4 \frac{|\bm{c}|^2 +  |\bm{b} - \bm{c}|^2}{|J_{\bm{\phi}}|}
\int_{\hat T} |{\nabla v}|^2 |J_{\bm{\phi}}| \, \dd \eta \dd \xi
\lesssim \frac{|\bm{c}|^2 +  |\bm{b} - \bm{c}|^2}{|J_{\bm{\phi}}|} |v|_{1,T_e}^2.
\end{equation}
As the two angles adjacent to $e$ are nonobtuse, $|\bm{c}|^2 
+|\bm{b} - \bm{c}|^2 \leq h_e^2 + 2l_e^2$, and~\eqref{eq:trlh} holds by the 
following estimate:
\[
\left(|\bm{c}|^2 +  |\bm{b} - \bm{c}|^2 \right) |J_{\bm{\phi}}|^{-1} \lesssim 
(h_e^2 + l_e^2)/(h_e l_e).
\]
Now when \hyperref[asp:A]{\bf A2} is met with $0<\gamma\leq 1$, we simply 
let $l_e = \gamma h_e$, and~\eqref{eq:tr-1/2} follows.
\end{proof}

\begin{lemma}[a trace inequality for an edge in a polygon]
\label{lemma:tr-l2}
If on a polygon $K$, an edge $e\subset \partial K$ has the height $l_e$ 
defined in~\eqref{eq:local-le}, and $T_e$ is the triangle with base $e$ and 
height $l_e$, then the following 
trace inequality holds:
\begin{equation}
\label{eq:tr-l2}
\norm{v}_{0,e} \lesssim l_e^{-1/2} \norm{v}_{0,T_e} 
+ \bigl(l_e^{1/2} + l_e^{-1/2}h_e\bigr) \norm{\nabla v}_{0,T_e}.
\end{equation}
Furthermore, if the edge $e$ satisfies \hyperref[asp:A]{\bf A2} with 
constant $0<\gamma\leq 1$:
\begin{equation}\label{eq:tr-h}
\norm{v}_{0,e} \lesssim \gamma^{-1/2}\left( h_e^{-1/2} \norm{v}_{0,K} + 
h_e^{1/2} \norm{\nabla v}_{0,K}\right).
\end{equation}
\end{lemma}

\begin{proof}
Consider a reference triangle $\widehat{T}_e$ with base $e$, height of length $h_e$, and the two angles adjacent to $e$ being acute. Let the 
base align with the horizontal $\hat{x}$-axis and choose the origin as the 
projection of vertex of $T_e$ opposite to $e$. It is well known that 
(\cite[Theorem 1.5.1.10]{Grisvard1985},~\cite[Appendix Lemma A.3]{Wang2014wgmixed}) 
for any $\widehat{v}\in H^1(\widehat{T}_e)$,
\begin{equation}
\norm{\widehat{v}}_{0,e}^2 \lesssim h_e^{-1} \norm{\widehat{v}}_{0,\widehat{T}_e}^2
+ h_e \bigl\| \widehat{\nabla}\widehat{v} \bigr\|_{0,\widehat{T}_e}^2.
\end{equation}
The mapping $\bm{\psi}: \hat{x}\mapsto x = \hat{x}$, $\hat{y}\mapsto y = l_e 
\hat{y}/h_e$ maps $\widehat{T}_e$ to $T_e$, which has base length $h_e$ and height 
$l_e$. Equation \eqref{eq:tr-l2} follows from a straightforward change of variable 
computation:
\begin{equation}
\norm{v}_{0,e}^2 = \norm{\widehat{v}}_{0,e}^2
\lesssim l_e^{-1} \norm{v}_{0,T_e}^2 
+ h_e^2 l_e^{-1} \norm{\partial_x v}_{0,T_e}^2
+ l_e \norm{\partial_y v}_{0,T_e}^2.
\end{equation} 
\eqref{eq:tr-h} then follows from the same argument for~\eqref{eq:tr-1/2} 
in the previous lemma:
\begin{align*}
\norm{v}_{0,e}&  \lesssim (\gamma h_e)^{-1/2} \norm{v}_{0,T_e} + 
\gamma^{-1/2}h_e^{1/2}  \norm{\nabla v}_{0,T_e}
\\
&\leq  \gamma^{-1/2}\left( h_e^{-1/2} \norm{v}_{0,K} + 
h_e^{1/2} \norm{\nabla v}_{0,K}\right).
\end{align*}
\end{proof}
\begin{remark}\rm 
A short edge with length $h_e\ll h_K$ is allowed provided that it satisfies 
\hyperref[asp:A]{\bf A2}. For a long edge that only supports a triangle 
with a short height, i.e., $l_e \ll h_e\eqsim h_K$, this trace inequality does not 
imply any meaningful estimate and this is one of the main difficulties for the 
anisotropic error estimate.
$\Box$
\end{remark}

\section*{Appendix B: Poincar\'{e}--Friedrichs Inequalities}
\label{appendixB}
In this appendix, we review Poincar\'{e}--Friedrichs inequalities with a constant 
depending only on the diameter of the domain but not on the shape. 

\begin{lemma}[Poincar\'{e} inequality with average zero in a convex domain 
\cite{Payne1960optimal}]
\label{lemma:poincare-convex}
Let $K \subset \mathbb R^d$ be a convex domain with a continuous boundary. For any 
$v\in H^1(K)$,  
\begin{equation}\label{poincare}
\norm{v - \bar{v}^{K}}_{0,K} \leq \frac{h_K}{\pi} \norm{\nabla v}_{0,K}. 
\end{equation} 
\end{lemma}

\begin{lemma}[Poincar\'{e} inequality with average zero on a subset] 
\label{lemma:poincare-subset}
Let $S\subset \mathbb{R}^2$ be a convex domain, and a nondegenerate $\omega 
\subset S$, then for any $v\in H^1(S)$,
\begin{equation}
\label{eq:poincare-subset}
\norm{v - \bar{v}^{\omega}}_{0,S} \leq  
\left ( 1+ \frac{|S|^{\frac12}}{|\omega|^{\frac12}}\right) \frac{h_{S}}{\pi} 
\norm{\nabla v}_{0,S}.
\end{equation}
\end{lemma}
\begin{proof}
A straightforward argument exploiting~\eqref{poincare} is as follows:
\begin{equation*}
\norm{v - \bar{v}^{\omega}}_{0,S} \leq 
\norm{v - \bar{v}^{S}}_{0,S} + \norm{\bar{v}^{S} - \bar{v}^{\omega}}_{0,S}
\leq \frac{h_S}{\pi} \norm{\nabla v}_{0,S} 
+ \frac{|S|^{\frac12}}{|\omega|^{\frac12}} \norm{v-\bar{v}^{S}}_{0,\omega}.
\end{equation*}
Then the estimate~\eqref{eq:poincare-subset} follows directly from 
$\norm{v-\bar{v}^{S} }_{0,\omega} \leq \norm{v-\bar{v}^{S}}_{0,S}$ and 
\eqref{poincare}.
\end{proof}

\begin{lemma}[Poincar\'{e} inequality with average zero on a subset of boundary]
\label{lemma:poincare-bd}
Let $K\subset \mathbb{R}^2$ be a Lipschitz polygon and $\Gamma\subset \partial K$ be 
a connected subset, then for any $v\in H^1(K)$ 
such that $\overline{v}^{\Gamma} = 0$, the following inequality holds:
\begin{equation}
\label{eq:poincare-bd}
\norm{v}_{0,\Gamma} \leq h_{\Gamma}^{1/2} |v|_{\frac12,\Gamma}.
\end{equation}

\end{lemma}
\begin{proof}
First inserting the $\overline{v}^{\Gamma}$ into the $L^2$-norm yields:
\begin{equation}
\norm{v}_{0,\Gamma}^2 
= \int_{\Gamma}  \Bigg(\frac{1}{|\Gamma|}\int_{\Gamma}  
\big(v(\bm{x})-v(\bm{y})\big) \dd s(\bm{y})\Bigg)^2\dd s(\bm{x}).
%\end{aligned}
\end{equation}
Inserting terms to match the $1/2$-seminorm yields the estimate 
\eqref{eq:poincare-bd}:
\begin{equation}
\norm{v}_{0,\Gamma}^2 = \frac{1}{|\Gamma|} \int_{\Gamma}  \int_{\Gamma} 
\frac{\big\vert v(\bm{x})-v(\bm{y})\big\vert^2}{|\bm{x}-\bm{y}|^2} 
\cdot|\bm{x}-\bm{y}|^2\, \dd s(\bm{y})\dd s(\bm{x})
\leq \frac{\max\limits_{x,y\in \Gamma}|\bm{x}-\bm{y}|^2}{|\Gamma|} 
|v|_{\frac12,\Gamma}^2.
\end{equation}
\phantom{;}
\end{proof}

\begin{lemma}[Poincar\'{e} inequality for continuous and piecewise polynomials 
in\\ $|\cdot|_{\frac12,\mathcal{E}_K}$]
\label{lemma:poincare-brokenhalf} 
Let $K\subset \mathbb{R}^2$ be a Lipschitz polygon 
and $\Gamma\subset \partial K$ be a connected subset; then for any $v\in 
B_p(\partial K)\subset C^0(\partial K)$ such that $\overline{v}^{\Gamma} = 0$, the 
following inequality holds:
\begin{equation}\label{eq:poincare-brokenhalf}
\norm{v}_{0,e} \lesssim n_{\mathcal{E}_K}^{1/2}  h_{e}^{1/2} 
|v|_{\frac12,\mathcal{E}_K}.
\end{equation}
\end{lemma}
\begin{proof}
First let $v\in B_1(\partial K)$, the mean value 
theorem implies that $v(\bm{\xi}) = 0$ for some $\bs \xi \in \Gamma$. Without loss of 
generality, $\bm{\xi}$ is assumed to be different with any given vertex on $\partial K$. 
Denote $\Gamma_{\bm{x}}$ as the curve along $\partial K$ from $\bm{\xi}$ to $\bm{x}$ 
for any $\bm{x}\in e$:
\begin{equation}\label{eq:error-2}
v(\bm{x}) = \int_{\Gamma_{\bm{x}}} \nabla v\cdot \bm{t}\,\dd s
\leq \sum_{\substack{e\subset \partial K,\\ e\cap \Gamma_{\bm{x}} \neq \varnothing 
}}\int_{e} |\partial_e v|\,\dd s
= \sum_{\substack{e\subset \partial K,\\ e\cap \Gamma_{\bm{x}} \neq \varnothing 
}} |v(\bm{b}_e) - v(\bm{a}_e)|.
\end{equation}
As a result, integrating on $e$ yields the estimate~\eqref{eq:poincare-brokenhalf} 
for $v\in B_1(\partial K)$:
\begin{equation}
\norm{v}_{0,e}^2 \leq  n_{\mathcal{E}_K}\, h_e
\sum_{\substack{e'\subset \partial K,\\ e'\cap \Gamma_{\bm{x}} \neq \varnothing 
}}
\,|v(\bm{b}_{e'}) - v(\bm{a}_{e'})|^2 = n_{\mathcal{E}_K} h_e
\sum_{e'\cap\Gamma\neq \varnothing} |v|_{\frac12,e'}^2.
\end{equation}
Now for $v\in B_p(\partial K)$, the point $\bm{\xi}$, where 
$v$ vanishes, is treated as an artificial new vertex on $\partial K$. With this 
additional vertex $\bm{\xi}$ on $\partial K$, let $v_I\in B_1(\partial K)$ be $v$'s 
linear nodal interpolation, and $\mathcal{E}'_K$ be the collection of the edges with 
two new edges, which have $\bm{\xi}$ as an end point, replacing the one edge in 
$\mathcal{E}_K$. Then by a standard Bramble--Hilbert estimate 
(\cite[Proposition 6.1]{Dupont-Scott1980}) and the same argument above for $v_I$, it 
holds that
\begin{equation}
\norm{v}_{0,e} \leq \norm{v-v_I}_{0,e} + \norm{v_I}_{0,e} \lesssim
h_e^{1/2} |v|_{\frac12,e} + (n_{\mathcal{E}_K}+1)^{1/2} h_{e}^{1/2} 
|v_I|_{\frac12,\mathcal{E}'}.
\end{equation}
By a standard inverse inequality and the interpolation error estimates,
\begin{equation}\label{eq:error-3}
|v_I|_{\frac12,\mathcal{E}'_K} \leq |v|_{\frac12,\mathcal{E}_K'} + 
|v-v_I|_{\frac12,\mathcal{E}'_K}
\lesssim |v|_{\frac12,\mathcal{E}'_K} +
\left(\sum_{e \in \mathcal{E}'_K} h_e^{-1}\norm{v-v_I}_{0,e}^2\right)^{1/2}
\lesssim |v|_{\frac12,\mathcal{E}'_K}.
\end{equation}
The lemma follows immediately from the fact that $|v|_{\frac12,\mathcal{E}'_K}\leq 
|v|_{\frac12,\mathcal{E}_K}$ by the definition of the $1/2$-seminorm 
\eqref{eq:norm-half}.
\end{proof}

\begin{corollary}[Equivalence between $\norm{\cdot}_{\infty,\partial K}$ and 
$|\cdot|_{\frac12,\mathcal{E}_K}$]
\label{corollary:normeq-infhalf} The following 
norm equivalence holds with a constant depending only on $n_{\mathcal{E}_K}$ and $p$:
\begin{equation}
\norm{v}_{\infty,\partial K} \lesssim 
|v|_{\frac12,\mathcal{E}_K} \lesssim \norm{v}_{\infty,\partial K}, \qquad  v\in 
B_p(\partial K),\; \overline{v}^{\partial K} = 0.
\end{equation}
\end{corollary}
\begin{proof}
The first inequality is a direct consequence of Lemma 
\ref{lemma:poincare-brokenhalf} and a standard inverse estimate on each edge 
$\norm{v}_{\infty,e}\lesssim h_e^{-1/2} \norm{v}_{0,e}$. 
For the second one, one can use a similar argument with~\eqref{eq:error-3}. Let 
$v_I$ be the linear interpolant of $v$ on $\partial K$
\begin{equation*} 
|v|_{\frac12,\mathcal{E}_K}\leq |v_I|_{\frac12,\mathcal{E}_K}
+ |v-v_I|_{\frac12,\mathcal{E}_K} \lesssim 
\left\{\sum_{e\subset \partial K} \!\Big(|v_I(\bm{b}_{e}) - v_I(\bm{a}_{e})|^2
+ h_e^{-1}\norm{v-v_I}_{0,e}^2\Big)\right\}^{1/2}.
\end{equation*}
The corollary follows by $v=v_I$ at the vertices, and
$h_e^{-\frac12}\norm{v-v_I}_{0,e} \lesssim \norm{v}_{\infty,e}$.
\end{proof}

\begin{remark}[difference between estimates~\eqref{eq:poincare-bd} and 
\eqref{eq:poincare-brokenhalf}]\rm
We remark that for an $H^{\frac12}(\partial K)$ function,~\eqref{eq:poincare-bd} is 
an estimate in the $1/2$-seminorm on the whole $\partial K$. While for a VEM 
function which is a piecewise polynomial on $\partial K$, a more delicate 
Poincar\'{e} inequality can be obtained in broken $1/2$-seminorm when the zero 
average is imposed. $\Box$
\end{remark}

For an edge $e\subset \partial K$, construct a trapezoid $K_e\subset K$ with base $e$ from $T_e$, which is the triangle with height $l_e$ defined 
in~\eqref{eq:local-le}, by connecting the 
midpoints of the two edges adjacent to $e$ in $T_e$.
%can be defined as follows: 
%\begin{equation}\label{eq:Ke}
%K_e = \bigl\{\bm{x} : \bm{x} = \bm{x}_e + t (\bm{x}_{e'} - \bm{x}_e), 
%t\in (0,1), \; \bm{x}_e \in e, \bm{x}_{e'} \in e' \bigr\}.
%\end{equation}
With a slight abuse of notation, we still denote the height of $K_e$ as $l_e$ for base $e$.

In the following proof, one can always start with a rectangle with side $h_e$ and 
height $l_e$. The general trapezoid can be transformed into the square by 
considering a parametrization (see Figure \ref{fig:anielem}(c)) for a quadrilateral 
with 
vertices $\bm{a}_i$ 
($i=1,\dots,4$). Without loss of generality, $\bm{a}_1$ is assumed to be the 
origin. For any $\bm{x}\in K$, and $(\xi,\eta)\in \widehat{S}:= (0,1)^2$, consider 
the following bilinear mapping $\bm{\psi}: (\xi,\eta)\mapsto \bm{x}$:
\begin{equation}
\label{eq:para-quad}
\bm{x} = \bm{\psi}(\xi,\eta) 
= (1-\xi) \eta\, \bm{a}_3 + \xi(1-\eta)\,\bm{a}_2 + \xi\eta \,\bm{a}_4.
\end{equation}
If the two opposite angles, sharing $e$ as one side, are uniformly bounded above and 
below, and the top edge has comparable length with the base $e$, then it is 
straightforward to verify that
\begin{equation}
\label{eq:para-dd}
\|\partial_{\xi}\bm{x}\|_{\infty, K}\lesssim h_e, \quad \|\partial_{\eta}\bm{x} 
\|_{\infty, K}\lesssim l_e, \quad |J_{\bm{\psi}}| \eqsim l_eh_e.
\end{equation}

\begin{lemma}[Poincar\'{e} inequality with average zero on a boundary edge]
\label{lemma:poincare-e}
Let $K\subset \mathbb{R}^2$ be a Lipschitz polygon, and let $e\subset \partial K$ be 
an edge satisfying the condition: there exists a trapezoid $K_e\subset K$ 
of height $l_e$ and base $e$, with two angles adjacent to $e$ uniformly bounded 
above and below. Then the following 
inequality holds:
\begin{equation}
\label{eq:poincare-eA}
\norm{v - \overline{v}^e}_{0,K_e} \lesssim (l_e h_e)^{1/2} |v|_{\frac12,e} + l_e 
\norm{\nabla v}_{0,K_e}.
\end{equation}
Moreover, if $K$ is convex, then
\begin{equation}
\label{eq:poincare-eR}
\norm{v - \overline{v}^e}_{0,K} \lesssim \frac{|K|^{1/2}}{|K_e|^{1/2}} \, h_{K} 
\norm{\nabla v}_{0,K}.
\end{equation}
\end{lemma}
\begin{proof}
Applying the parametrization $\bm{x} = \bm{\psi}(\xi,\eta)$ in 
\eqref{eq:para-quad} for any $\bm{x}\in K_e$, define $\widehat v(\xi, \eta) := 
v\big(\bm{\psi}(\xi,\eta)\big)$ for $(\xi, \eta)\in \widehat{S}$. By the fundamental 
theorem of calculus,
\begin{equation}
\widehat v(\xi,\eta) - \overline{v}^{e} 
= \widehat v(\xi,0) - \overline{v}^{e}
+ \int^{\eta}_{0} \partial_{\eta} 
\widehat v(\xi,\tau)\,\dd\tau.
\end{equation}
By the relation~\eqref{eq:para-dd} and Young's inequality, 
\begin{align*}
\norm{v - \overline{v}^e}_{0,K}^2 
\lesssim l_e h_e \norm{\widehat v - \overline{v}^e}_{0,\widehat{S}}^2
\lesssim l_e h_e\left ( \norm{\widehat v - \overline{v}^e}_{0,\hat e}^2 
+ \|\partial_{\eta}\widehat v\|_{0,\widehat{S}}^2 \right ),
\end{align*}
where $\hat e = \{\xi\in (0,1), \eta = 0\}$ is the pre-image of $e$ in the reference 
square. Using~\eqref{eq:para-dd} and 
Poincar\'e inequality~\eqref{eq:poincare-eA}, the first term can be bounded by
$$
h_e\norm{\widehat v - \overline{v}^e}_{0,\hat e}^2 = \norm{v - 
\overline{v}^e}_{0,e}^2\le h_e|v|_{\frac12,e}^2.
$$
The estimate~\eqref{eq:poincare-eA} then follows from:
$$
l_e h_e \|\partial_{\eta}\widehat v\|_{0,S}^2\lesssim \int_S |\nabla v(\bm{x})|^2 
|\partial_{\eta}\bs x|^2 |J_{\bm{\psi}}| \dd \xi \dd \eta 
\lesssim l_e^2\,\| \nabla v\|_{0,K_e}^2.
$$
To obtain~\eqref{eq:poincare-eR}, one can apply the trace inequality~\eqref{eq:trlh} 
to get
\begin{equation}
\label{eq:poincare-eA-1}
\norm{v- \overline{v}^{e}}^2_{0,K_e}\lesssim 
l_e h_e |v|_{\frac12,e}^2 + l_e^2 \norm{\nabla v}_{0,K_e}^2 \lesssim 
(l_e+h_e)^2\|\nabla v\|_{0,K}^2.
\end{equation}
To bridge the inequality from $K_e$ to $K$, using the triangle inequality, 
Poincar\'e inequality~\eqref{poincare}, $h_e\eqsim h_K, l_e \leq h_K$, and the 
estimate above yield
\begin{align}
\norm{v- \overline{v}^{e}}_{0,K}
& \leq \norm{v- \overline{v}^K}_{0,K} + \norm{\overline{v}^K- 
\overline{v}^{K_e}}_{0,K}
+ \norm{\overline{v}^{K_e}-\overline{v}^{e}}_{0,K}
\nonumber \\
&\lesssim h_K\norm{\nabla v}_{0,K}
+\frac{|K|^{1/2}}{|K_e|^{1/2}}\left ( \norm{v- \overline{v}^K}_{K_e} + \norm{v- 
\overline{v}^{e}}_{K_e}\right )
\nonumber \\
&\lesssim \frac{|K|^{1/2}}{|K_e|^{1/2}} \, h_{K} 
\norm{\nabla v}_{0,K}.
\end{align}

\end{proof}

\section*{Acknowledgment}
The authors appreciate the anonymous reviewers for valuable suggestions and 
insightful comments, which improved an early version of the paper.

\end{document}